\documentclass[11pt,reqno]{amsart}
\usepackage{amssymb}
\usepackage{mathptmx}
\usepackage{mathtools}
\usepackage{mathrsfs}
\usepackage[export]{adjustbox}
\usepackage[all]{xy}
\usepackage{stmaryrd}
\usepackage{fancyhdr}
\usepackage{hyperref}
\usepackage{tikz-cd}
\usepackage{cite}
\usepackage{amsmath,amsfonts}
\usepackage{graphicx}
\usepackage{wrapfig}
\usepackage{array}
\usepackage{amsthm}
\usepackage[left=1.2in, right=1.2in, top=1in, bottom=1in]{geometry}
\usepackage{etoolbox}
\patchcmd{\section}{\scshape}{\bfseries}{}{}
\makeatletter
\renewcommand{\@secnumfont}{\bfseries}
\makeatother
\xyoption{all}
\usepackage{secdot}

\theoremstyle{plain}
\DeclareMathOperator{\id}{\textrm{id}}

\newcommand{\Aff}{\mathbb{A}}

\newcommand{\FF}{\mathbb{F}}    
\newcommand{\NN}{\mathbb{N}} 
\newcommand{\ZZ}{\mathbb{Z}} 
\newcommand{\Vect}{\operatorname{Vect}}  

\newcommand{\Rep}{\operatorname{Rep}}

\newcommand{\nil}{\operatorname{nil}}  
\newcommand{\iso}{\operatorname{iso}}
\newcommand{\Hom}{\operatorname{Hom}} 
  
\newcommand{\supp}{\operatorname{supp}} 
   
\newcommand{\End}{\operatorname{End}}
\newcommand{\Fun}{\operatorname{Fun}}
\newcommand{\Obj}{\operatorname{Obj}}    
\newcommand{\Symm}{\operatorname{Symm}}
\newcommand{\C}{\mathcal{C}}   
\newcommand{\E}{\mathcal{E}}    
 
\newcommand{\mcO}{\mathcal{O}}  
\newcommand{\mcI}{\mathcal{I}}  
\newcommand{\mfM}{\mathfrak{M}} 
\newcommand{\mfI}{\mathfrak{I}} 
\newcommand{\mfE}{\mathfrak{E}}
 
\newcommand{\M}{\operatorname{M}} 
\newcommand{\MM}{\operatorname{SM}}

\newcommand{\ab}{\operatorname{ab}} 
 
\newcommand{\mon}{\operatorname{mon}}

\input xy
\xyoption{all}

\theoremstyle{definition}
\newtheorem{mydef}{\textbf{Definition}}[section]
\newtheorem{myeg}[mydef]{\textbf{Example}}

\newtheorem{rmk}[mydef]{\textbf{Remark}}
\newtheorem{construction}[mydef]{\textbf{Construction}}

\theoremstyle{plain}

\newtheorem*{nothma}{\textbf{Theorem A}}
\newtheorem*{nothmb}{\textbf{Theorem B}}

\newtheorem{mythm}[mydef]{\textbf{Theorem}}
\newtheorem{lem}[mydef]{\textbf{Lemma}}
\newtheorem{pro}[mydef]{\textbf{Proposition}}

\newtheorem{cor}[mydef]{\textbf{Corollary}}
 
\newtheorem*{convention}{\textbf{Convention}}

\tikzset{main node/.style={circle,fill=black,draw,minimum size=0.3cm,inner sep=0pt},
}

\begin{document}

	\title{On Semisimple Proto-Abelian Categories Associated to Inverse Monoids}  
	%Semisimple Proto-Abelian Categories and Representation Theory over the Field with One Element
	%Representations of Matrix Bialgebras over the Field with One Element 
	%Representations of Submonomial Matrix Monoids over Pointed Abelian Groups

	\author{Alexander Sistko}
	\address{SUNY New Paltz}
	\curraddr{}
	\email{sistkoa@newpaltz.edu}
	
	\makeatletter
	\@namedef{subjclassname@2020}{%
		\textup{2020} Mathematics Subject Classification}
	\makeatother
	

	%\subjclass[2010]{}
	%\keywords{}
	%\date{\today}
	
	\dedicatory{}

	\maketitle
	
	%\tableofcontents
	
\begin{abstract} 
Let $G$ be a finite abelian group written multiplicatively, with $\hat{G} = G\sqcup \{0\}$ the pointed abelian group formed by adjoining an absorbing element $0$. There is an associated finitary, proto-abelian category $\Vect_{\hat{G}}$, whose objects can be thought of as finite-dimensional vector spaces over $\hat{G}$. The class of $\hat{G}$-linear monoids are then defined in terms of this category. In this paper, we study the finitary, proto-abelian category $\operatorname{Rep}(M,\hat{G})$ of finite-dimensional $\hat{G}$-linear representations of a $\hat{G}$-linear monoid $M$. Although this category is only a slight modification of the usual category of $M$-modules, it exhibits significantly different behavior for interesting classes of monoids. Assuming that the regular principal factors of $M$ are objects of $\operatorname{Rep}(M,\hat{G})$, we develop a version of the Clifford-Munn-Ponizovski\u i Theorem and classify the $M$ for which each non-zero object of $\operatorname{Rep}(M,\hat{G})$ is a direct sum of simple objects. When $M$ is the endomorphism monoid of an object in $\Vect_{\hat{G}}$, we discuss alternate frameworks for studying its $\hat{G}$-linear representations and contrast the various approaches.  
\end{abstract}  

%%%%%%%%%%%%%%%%%%%%%%%%%%%%%%%%%%%%%%%%%%%%%%%%
%%%%%%%%%%%%%%%%%%(INTRODUCTION)%%%%%%%%%%%%%%%%%%%%% 
%%%%%%%%%%%%%%%%%%%%%%%%%%%%%%%%%%%%%%%%%%%%%%%%% 
	
\section{Introduction}  

%Stylistic changes: 
%1) "submonomial matrix" instead of "monomial matrix." \MM gets replaced by \SM. 
%2) Replace [n] with \widehat{[n]} and [n]\setminus\{0\} with [n]. 
% 3) \Vect(\hat{G}) is replaced with \Vect_{\hat{G}}

This article represents a synthesis of the theories of finite semigroups and mathematics over the field with one element. Some cross-pollination between these fields of study already exists: for instance, researchers working in Hall algebras of finitary proto-exact categories have used actions of semigroups and monoids on several occasions \cite{szczesny2014hall, szczesny2018hopf,BeersSzcz2019}, and representations of finite semigroups over semirings were studied in \cite{IzhRhoStein2011}. Furthermore, the literature on semigroups is extensive and rich, with many ideas mirroring those considered over the field with one element \cite{CliffPrest1961, Tully1964}. However, the emphasis of much recent work seems to have been on representations which arise as homomorphisms into transformation monoids, rather than those into symmetric inverse monoids. In this article, the author hopes to make the case for the distinctness and significance of the latter. Additionally, they hope that this work will draw further attention to the structure and representation theory of semigroups, which others have noted has sometimes felt insulated from other sub-disciplines of algebra \cite{GanMazStein2009,Steinberg2016}. In particular, this work represents an attempt to elucidate how important results from semigroup theory can be developed within the framework of finitary proto-abelian categories through a slight modification of the usual functorial approach. As evidence of its usefulness, this framework recovers important classes of semigroups from a natural categorical context, and characterizes them via representation theory over the field with one element.    

It should be noted, though, that the initial motivation for this article was much more humble. Originally, the author wanted to develop a version of Wedderburn-Artin Theory for representation theory over the field with one element. Recall that an Artinian ring $R$ is said to be semisimple if every $R$-module is isomorphic to a direct sum of simple representations. Wedderburn-Artin Theory asserts that $R$ is semisimple if and only if  
\[ 
R \cong \prod_{i=1}^d{M_{n_i}(\Delta_i)}
\] 
for some positive integers $n_i$ and division rings $\Delta_i$. In the important case that $R$ is a finite-dimensional algebra over a field $k$, the $\Delta_i$ are all finite-dimensional division algebras over $k$. In particular, if $k$ is finite, then Wedderburn's Little Theorem says that the $\Delta_i$'s are all field extensions of $k$ of finite degree. Thus, the semisimple algebras over finite fields are products of full matrix rings over finite fields. Of course, every field contains at least two elements: however, starting with a seminal article from Tits \cite{tits1956analogues}, some have found it fruitful to treat certain objects in mathematics as if they were defined over a fictional ``field with one element,'' here denoted $\FF_1$. As it stands, mathematics over the field with one element is a diverse area of contemporary research, see \cite{LorscheidFun} for a survey over its applications within geometry and number theory. For the present article, we focus instead on the algebraic and combinatorial aspects of the subject, whose typical objects of study are finitary, proto-exact categories which act as combinatorial analogs to categories of representations, sheaves, etc. \cite{szczesny2011representations, szczesny2012hall, szczesny2014hall, szczesny2018hopf, eppolito2017hopf, eppolito2018proto, jun2020quiver, JunSzczToll2023, JunSis2023, KleinauPreprint, JarraLorscheidVitalPreprint}. These categories are rarely additive, and often exhibit behavior quite different from their analogs defined over fields. 

Within this context, there is a category of finite-dimensional vector spaces over $\FF_1$, denoted $\Vect_{\FF_1}$. Its objects are finite pointed sets $(V,0_V)$, and its morphisms are functions $f : (V,0_V) \rightarrow (W,0_W)$ satisfying $f(0_V) = 0_W$ and $|f^{-1}(y)| \le 1$ for all $y \neq 0_W$. Thus, one can conceive of representation theory over $\FF_1$ as the study of certain categories of functors into $\Vect_{\FF_1}$. 

For each positive integer $n$, there is a unique $n$-dimensional object up to isomorphism, namely the $(n+1)$-element set
\[
\widehat{[n]} = \{0,1,\ldots , n\}.
\]
Its endomorphism monoid $I_n$ should be an $\FF_1$-analog to a full matrix ring over a field. Indeed, morphisms in $I_n$ may be naturally identified with $n\times n$ submonomial matrices, i.e. matrices with entries in $\{0,1\}$ with the property that each row and column contains at most one non-zero entry. Our choice of notation for $I_n$ comes from the semigroup literature, where it is known as the inverse symmetric monoid of partial bijections on the set $[n] = \{1,\ldots , n\}$ \cite{CliffPrest1961,Steinberg2016}. Note that our definition still yields a monoid if we replace the set $\{0,1\}$ with any monoid, insinuating that we may wish to consider more general sets of coefficients than $\FF_1$. This indeed works, and for the purposes of the present article we retain a reasonable theory if we replace $\FF_1$ with any finite pointed abelian group. This has precedence within the literature on $\FF_1$, where the field extension of $\FF_1$ of finite degree $d$ corresponds to the group of $d^{th}$ roots of unity \cite{LorscheidFun}. Similar ideas have also been considered within the semigroup literature \cite{CliffPrest1961,Tully1964}. More precisely, if $G$ is a finite abelian group written in multiplicative notation, we get a finitary, proto-abelian category $\Vect_{\hat{G}}$ of finite-dimensional vector spaces over the pointed abelian group $\hat{G} = G\sqcup \{0\}$. Note that the choice of $G = \{ 1\}$ recovers $\Vect_{\FF_1}$ directly. The endomorphism monoid of the standard $n$-dimensional object of $\Vect_{\hat{G}}$ now yields a monoid $I_n(\hat{G})$ which naturally contains $I_n = I_n(\FF_1)$ as a submonoid. The monoid $I_n(\hat{G})$ is an example of what we term a $\hat{G}$-linear monoid, which is a non-additive analog to a finite-dimensional associative algebra over a field. To any $\hat{G}$-linear monoid $M$, we now define a finitary, proto-abelian category $\Rep(M,\hat{G})$ of $\hat{G}$-linear representations. Objects in this category are certain functors $M \rightarrow \Vect_{\hat{G}}$, treating $M$ as a category with one object. 

It is now natural to ask what $\Rep(M,\hat{G})$ looks like for various $M$. In analogy to Wedderburn-Artin Theory: for which $M$ is $\Rep(M,\hat{G})$ semisimple, in the sense that all non-zero objects in $\Rep(M,\hat{G})$ are direct sums of simple objects? In the case of Artinian rings, allowing the ring to act on itself via left multiplication serves as a crucial step towards classification. Unfortunately, left multiplication of $M$ on itself usually does not yield an object in $\Rep(M,\hat{G})$. A more appropriate condition is to require that the regular principal factors of $M$ be $\hat{G}$-linear representations of $M$ under left translation: we call such monoids left-inductive, with right-inductive monoids being defined similarly. For this class of monoids, there is a bijection between simple $\hat{G}$-linear representations of $M$ and simple $\hat{G}$-linear representations of its maximal subgroups. More precisely, if $J$ is a regular $\mathcal{J}$-class containing the idempotent $e$, and $G_J$ is the maximal subgroup $eMe\cap J$, we obtain an induction functor $(-){\uparrow_e} : \Rep(\hat{G}_J,\hat{G}) \rightarrow \Rep(M,\hat{G})$ and restriction functor $(-){\downarrow_e} : \Rep(M,\hat{G}) \rightarrow \Rep(\hat{G}_J,\hat{G})$ which satisfy the following Clifford-Munn-Ponizovski\u i Theorem: 

\begin{nothma}[Theorem \ref{t.CMP}] 
Let $M$ be a finite, $\hat{G}$-linear, left inductive monoid. Fix a non-zero idempotent $e \in M$, and let $J = \mathcal{J}_e$ be its associated regular $\mathcal{J}$-class. Then the following hold: 
\begin{enumerate}  
\item If $V \in \Rep(\hat{G}_J,\hat{G})$ is simple, then  
\[ 
N(V{\uparrow_e}) := \{x\otimes w \in V{\uparrow_e} \mid eM(x\otimes w) = 0 \} 
\] 
is the unique maximal submodule of $V{\uparrow_e}$. Furthermore, the simple representation $Q(V) = V{\uparrow_e}/N(V{\uparrow_e})$ has apex $J$ and $Q(V){\downarrow_e} \cong V$ in $\Rep(\hat{G}_J,\hat{G})$.  
\item If $V \in \Rep(M,\hat{G})$ is simple with apex $J$, then $W= V{\downarrow_e}$ is simple in $\Rep(\hat{G}_J,\hat{G})$. Furthermore,
\[ 
W{\uparrow_e}/N(W{\uparrow_e}) \cong V
\] 
in $\Rep(M,\hat{G})$.
\end{enumerate} 
In particular, there is a bijection between simple objects in $\Rep(M,\hat{G})$ with apex $J$ and simple objects in $\Rep(\hat{G}_J,\hat{G})$.
\end{nothma} 

This theorem is an addition to a line of similar results that exist for representations of finite semigroups over commutative rings \cite{GanMazStein2009} and commutative semirings \cite{IzhRhoStein2011}.

%Over a field $k$, $M_n(k)$ is actually simple: it has a unique simple representation $S$ up to isomorphism, and every representation of $M_n(k)$ is a direct sum of copies of $S$. In contrast, $\Rep(I_n(\hat{G},\hat{G})$ has many isomorphism classes of simple $\hat{G}$-linear representations. However, objects in $\Rep(I_n(\hat{G},\hat{G})$ still decompose uniquely into direct sums of simple representations, and so $\Rep(I_n(\hat{G}),\hat{G})$ is still a semisimple category. More precisely, we prove the following:  

It turns out that this class of monoids is well-behaved enough to admit a characterization in terms of semisimplicity. The main theorem we prove in this work is the following: 

\begin{nothmb}[Theorem \ref{t.main}] 
Let $M$ be a finite, $\hat{G}$-linear, left inductive monoid. Then the following are equivalent: 
\begin{enumerate} 
\item $\Rep(M,\hat{G})$ is semisimple. 
\item For all nilpotent ideals $J$ of $M$ and $V \in \Rep(M,\hat{G})$ we have $JV = 0$. 
\end{enumerate}  
In particular, if $M$ is also regular then $\Rep(M,\hat{G})$ is semisimple.
\end{nothmb}

The result above immediately implies that $\Rep(I_n(\hat{G}),\hat{G})$ is semisimple. In contrast, note that the finitary, proto-exact category of finite pointed $I_n(\hat{G})$-acts contains non-zero objects which are not direct sums of simples. Thus, we see that $\Rep(I_n(\hat{G}),\hat{G})$ retains more features of the representation theory of matrix rings than the category considered in \cite{szczesny2014hall}, while still exhibiting non-trivial differences with the classical theory over fields. To gain a better understanding of the relationship between the classical and non-additive theories, we next consider various means to endow $I_n(\hat{G})$ with extra structure. Treating $I_n(\hat{G})$ as a monoid endowed with a partial order, we recover a finitary proto-abelian category of representations which is simple (i.e. is semisimple and has only one simple object, up to isomorphism). Furthermore, scalar extensions of this algebra may naturally be interpreted as representations over a full matrix algebra. Treating $I_n(\hat{G})$ via the theory of affine schemes and bialgebras over the group algebra $\ZZ[G]$, we attain a category of representations which is not simple, but appears more restrictive than $\Rep(I_n(\hat{G}),\hat{G})$ itself. Treating $I_n$ as a functor of finite pointed abelian groups yields a similar result. The results of this section are not meant to provide an exhaustive treatment, but rather to build a starting point for future explorations. In particular, for the latter two approaches we only prove that each category of representations contains infinitely-many non-isomorphic objects and at least two simple objects. 

This paper is organized in follows: in Section \ref{s.cat}, we discuss the basics of finitary proto-abelian categories; in Section \ref{s.GLin}, we develop the theory of $\hat{G}$-linear monoids and their representation categories; in Section \ref{s.SS}, we prove semisimplicity results for $\hat{G}$-linear representations of $\hat{G}$-linear monoids; and in Section \ref{s.Alt} we discuss alternate formulations for the study of $\hat{G}$-linear representations of $I_n(\hat{G})$. Since this article draws on results from essentially independent communities of researchers, some results will probably be obvious to one community and less-known to the other. This applies equally well to the author himself. To the best of their abilities, the author has made an effort to cite the originators of known results: in cases where the author is unsure of who is responsible, they have indicated what the result seems most similar to. %The author has found the resources \cite{CliffPrest1961, GanMazStein2009} particularly helpful in this process of attribution. 
 
 \bigskip 

\textbf{Declaration of Funding}\hspace{0.1cm} The author received no outside funding for the present work.  

\bigskip 

\textbf{Acknowledgments}\hspace{0.1cm} The author wishes to thank the anonymous referee for several helpful suggestions. 

%%%%%%%%%%%%%% (FINITARY PROTO EXACT CATEGORIES) %%%%%%%%%%%%%
\section{Finitary Proto-Abelian Categories}\label{s.cat}

We recall the basics of proto-exact and proto-abelian categories, first defined in \cite{dyckerhoff2012higher}. We follow the treatment of Szczesny, see for instance \cite{JunSzczToll2023}. 

\begin{rmk}
All categories $\C$ considered in this article are assumed to be small. The set of morphisms $X \rightarrow Y$ in $\C$ will be denoted as $\C(X,Y)$, $\Hom_{\C}(X,Y)$, or $\Hom_{\hat{G}}(X,Y)$ in the case that $\C = \Vect_{\hat{G}}$ (defined below). The set of endomorphisms $\C(X,X)$ will alternatively be denoted $\End_{\C}(X)$, $\End_{\hat{G}}(X)$, etc. 
\end{rmk}

 %Our terminology mirrors that of 
\begin{mydef} 
Let $\C$ be a category with zero object $0$, equipped with two classes of morphisms $\mfM$ and $\mfE$. We say that $\C$ is a \emph{proto-exact category} with \emph{admissible monomorphisms} $\mfM$ and \emph{admissible epimorphisms} $\mfE$ if the following conditions hold: 
\begin{enumerate} 
\item For each object $X$ of $\C$, $0\rightarrow X$ is in $\mfM$ and $X\rightarrow 0$ is in $\mfE$.
\item $\mfM$ and $\mfE$ contain all isomorphisms and are closed under composition. 
\item The commutative square: 
\begin{equation}\label{eq.biC}
\begin{tikzcd} 
X \arrow[hook]{r}{i} \arrow[swap, two heads]{d}{p} & Y \arrow[two heads]{d}{p'} \\ 
X' \arrow[swap, hook]{r}{i'} & Y'
\end{tikzcd} 
\end{equation} 
with $i,i' \in \mfM$ and $p,p'\in \mfE$ is Cartesian (a pullback) if and only if it is co-Cartesian (a pushout). We say that such a diagram is bi-Cartesian.
\item A diagram in $\C$ of the form 
\begin{equation} 
\begin{tikzcd} 
 & Y \arrow[two heads]{d}{p'} \\ 
X' \arrow[swap, hook]{r}{i'} & Y'
\end{tikzcd} 
\end{equation} 
with $i' \in \mfM$ and $p' \in \mfE$ can be completed to a bi-Cartesian square as in \ref{eq.biC} with $i \in \mfM$ and $p \in \mfE$. 
\item A diagram in $\C$ of the form 
\begin{equation}
\begin{tikzcd} 
X \arrow[hook]{r}{i} \arrow[swap, two heads]{d}{p} & Y  \\ 
X'  &
\end{tikzcd} 
\end{equation}  
with $i \in \mfM$ and $p \in \mfE$ can be completed to a bi-Cartesian square as in \ref{eq.biC} with $i' \in \mfM$ and $p' \in \mfE$.
\end{enumerate}  
If $\mfM$ is the class of all monomorphisms in $\C$ and $\mfE$ is the class of all epimorphisms in $\C$, we say that $\C$ is \emph{proto-abelian}.
\end{mydef}  

\noindent See \cite{dyckerhoff2012higher,szczesny2014hall, jun2020toric, JunSzczToll2023} for examples of proto-exact and proto-abelian categories similar to those considered in the present work. For additional treatments, see \cite{Andre2009,ChenHNCat2010}.  \newline

Let $(\C,\mfM,\mfE)$ be a proto-exact category. An \emph{admissible short exact sequence} in $\C$ is a bi-Cartesian square of the form  
\begin{equation}
\begin{tikzcd} 
A \arrow[hook]{r}{i} \arrow[swap, two heads]{d}{} & B \arrow[two heads]{d}{p} \\ 
0 \arrow[swap, hook]{r}{} & C
\end{tikzcd}  
\end{equation}  
with $i \in \mfM$ and $p \in \mfE$. If the maps $i$ and $p$ are understood, we can write this succinctly as $A \hookrightarrow B \twoheadrightarrow C$. Two admissible short exact sequences $A\hookrightarrow B \twoheadrightarrow C$ and $A \hookrightarrow B'\twoheadrightarrow C$ are \emph{equivalent} if there exists a commutative diagram 
\begin{equation}
\begin{tikzcd} 
A \arrow[hook]{r}{i} \arrow[swap]{d}{\id} & B \arrow[two heads]{r}{p} \arrow[swap]{d}{\cong} & C \arrow[]{d}{\id} \\ 
A \arrow[swap, hook]{r}{i'} & B' \arrow[swap, two heads]{r}{p'}  & C \\
\end{tikzcd}   
\end{equation}  
The set of all equivalence classes of such sequences is denoted $\operatorname{Ext}_{\C}(C,A).$ 

\begin{mydef} 
A proto-exact category $(\C,\mfM,\mfE)$ is \emph{finitary} if $\Hom_{\C}(A,B)$ and $\operatorname{Ext}_{\C}(A,B)$ are finite sets for all objects $A$ and $B$ of $\C$.
\end{mydef}   

\begin{mydef} 
$\Vect_{\FF_1}$ is the category whose objects are finite pointed sets $(V,0_V)$, and whose morphisms $f : (V,0_V)\rightarrow (W,0_W)$ are functions $f : V\rightarrow W$ satisfying $f(0_V) = 0_W$ and $|f^{-1}(w)| \le 1$ for all $w \in W\setminus\{0_W\}$. Objects in $\Vect_{\FF_1}$ are called (finite-dimensional) vector spaces over the field with one element $\FF_1$. Morphisms in $\Vect_{\FF_1}$ are called $\FF_1$-linear maps. 
\end{mydef}  

In this article, it will often be necessary to take an object and ``make it pointed'' by adjoining a zero or absorbing element. We try to stick to the convention that a caret indicates adding a zero element, e.x. to distinguish an abelian group $G$ from the pointed abelian group $\hat{G}$ with $\hat{G}\setminus\{0\} = G$. Thus, we make the following simple (but important) definition.  

\begin{mydef} 
Let $n$ be a positive integer. Then $[n]$ denotes the standard $n$-element set $[n] = \{ 1, \ldots , n\}$ and $\widehat{[n]}$ denotes the $(n+1)$-element set $\widehat{[n]} = \{0,1,\ldots, n\}$.
\end{mydef}

We now recall some elementary properties of $\Vect_{\FF_1}$, see \cite{szczesny2011representations} for more details: 

\begin{enumerate} 
\item $\Vect_{\FF_1}$ is a finitary proto-abelian category. The zero object of $\Vect_{\FF_1}$ is the singleton set $\{0\}$.
\item If $V \in \Vect_{\FF_1}$, a subspace $W \le V$ is the same thing as a subset of $V$ containing $0_V$. The quotient space $V/W$ is defined as $V/W = (V\setminus W)\cup \{0_V\}$. 
\item The dimension of $V$ is its number of non-zero elements $\dim_{\FF_1}(V) = |V|-1$.  
\item Every morphism $f : V \rightarrow W$ in $\Vect_{\FF_1}$ has a kernel $f^{-1}(0_W)$ and a cokernel $W/f(V)$, and $f(V) \cong V/\ker(f)$ in $\Vect_{\FF_1}$. 
\item $\Vect_{\FF_1}$ possesses a symmetric monoidal structure $\oplus$ defined via 
\[ 
V\oplus W = V\sqcup W/\{0_V,0_W\}.
\] 
\item Every non-zero object in $\Vect_{\FF_1}$ is isomorphic to $\widehat{[n]} = \{ 0,1,\ldots , n\}$ for a suitable $n \in \NN$, and 
\[ 
\widehat{[n]} \cong\underbrace{\widehat{[1]}\oplus \ldots \oplus \widehat{[1]}}_{n-times}
\]
\end{enumerate} 

Vector spaces over $\FF_1$ admit a generalization which will be useful for this paper. Recall that if $M$ is a monoid with absorbing element, a \emph{finite $M$-act} is a finite pointed set $(X,0_X)$ together with a function $M\times X \rightarrow X$, $(m,x)\mapsto m\cdot x$ satisfying the following for all $m_1, m_2 \in M$ and $x \in X$:
\begin{enumerate} 
\item $(m_1m_2)\cdot x = m_1\cdot (m_2\cdot x)$.
\item $1\cdot x = x$.
\item $0\cdot x = 0_X$.
\end{enumerate}
 For $x \in X$, the \emph{$M$-orbit of $x$} is the set $Mx = \{ m\cdot x \mid m \in M\}$. The set $\{0_X\}$ is always an orbit of $X$. We say that $X$ is \emph{finitely-generated} if $X$ is the union of a finite set of $M$-orbits. If $X$ and $Y$ are $M$-acts, an \emph{$M$-equivariant map} is a function $f : X\rightarrow Y$ satisfying $f(0_X) = 0_Y$ and $f(m\cdot x) = m\cdot f(x)$ for all $m \in M$ and $x \in X$. If $M = \hat{G}$ is a finite pointed abelian group, we say that $X$ is \emph{free} if $g\cdot x = x$ and $x \neq 0_X$ implies $g = 1$. This is equivalent to the action of the group $G$ on the set $X \setminus \{0_X\}$ being free in the usual sense.

\begin{mydef} 
Let $G$ be a finite abelian group, written multiplicatively, with $\hat{G} = G\sqcup\{0\}$ the associated pointed abelian group. Then $\Vect_{\hat{G}}$ is the category whose objects are finitely-generated, free pointed $\hat{G}$-acts and whose morphisms are $\hat{G}$-equivariant maps which induce $\FF_1$-linear maps on orbits. More explicitly, this latter condition means that a $\hat{G}$-equivariant map $f : V \rightarrow W$ is a morphism if and only if for all $v_1, v_2 \in V$, $f(v_1) = f(v_2) \neq 0_W$ implies $\hat{G}v_1 = \hat{G}v_2$. Objects in $\Vect_{\hat{G}}$ are called (finite-dimensional) $\hat{G}$-vector spaces, and morphisms are called $\hat{G}$-linear maps.
\end{mydef} 

\begin{rmk} 
Note that $\Vect_{\FF_1} = \Vect_{\hat{1}}$, where $\hat{1} = \{0,1\}$ is the monoid induced from the trivial group.
\end{rmk} 

%%%%%%%%%%%%%% (THE CATEGORIES VECT(G) ) %%%%%%%%%%%%%%%%%%%
%\subsection{The Categories $\Vect(\hat{G})$}    

%\begin{mydef} 
%Let $G$ be a finitely-generated multiplicative abelian group, with associated monoid $\hat{G} := G \sqcup\{0\}$. Let $\Vect_{\hat{G}}$ denote the category of finitely-generated, free pointed $\hat{G}$-acts with morphisms given by $\hat{G}$-equivariant maps that induce $\FF_1$-linear maps on orbits.
%\end{mydef}  

\begin{rmk}\label{r.Gidem} 
$\hat{G}$-linear maps are actually $\FF_1$-linear. For instance, suppose that $f : V \rightarrow W$ is a morphism in $\Vect(\hat{G})$. If $f(v_1) = f(v_2) \neq 0$, then $\hat{G}v_1 = \hat{G}v_2$ by $\hat{G}$-linearity. But $v_1$ and $v_2$ must be non-zero, so $v_1 = gv_2$ for some $g \in G$. But then $f(v_1) = f(gv_2) = gf(v_2) = gf(v_1)$ and so $g=1$ by the freeness of the orbit $Gf(v_1)$. In particular, $\hat{G}$-linear idempotents either fix elements or map them to zero. %Indeed, if $E \in \End_{\hat{G}}(V)$ is idempotent and $v \in V$, then $E(Ev) = Ev$ implies that either $Ev = 0$ or $Ev$ and $v$ generate the same orbit. In the latter case we have $Ev = gv$ for some $g \in G$, but then $g^2 v = E^2v = Ev = gv$ implies that $g = 1$ and $Ev = v$. 
\end{rmk} 

\noindent We note two functors of interest between $\Vect_{\hat{G}}$ and $\Vect_{\FF_1}$: 
\begin{enumerate} 
\item The extension of scalars functor $\hat{G}\otimes_{\FF_1}- : \Vect_{\FF_1} \rightarrow \Vect_{\hat{G}}$ which carries $V$ to the $\hat{G}$-set with non-trivial $G$-orbits $\{ G\times \{v\} \mid v \in V\setminus\{0_V\} \}$ and $L : V \rightarrow W$ to the $\hat{G}$-linear map satisfying $(1,v) \mapsto (1,L(v))$ for all $v \in V\setminus\{0_V\}$.
\item The $G$-orbit functor $\mcO_G : \Vect_{\hat{G}} \rightarrow \Vect_{\FF_1}$ which carries an object $V$ to its set of $G$-orbits and a $\hat{G}$-linear map $L : V \rightarrow W$ to the $\FF_1$-linear map induced on $\mcO_G(V) \rightarrow \mcO_G(W)$.
\end{enumerate} 

We now use the extension of scalars functor to define a family of monoids which will play a central role in this article.

\begin{mydef}\label{d.I_nG}
Identifying $\Vect_{\FF_1}$ with $\Vect_{\hat{1}}$, we may apply the extension of scalars functor and define  
\[
I_n(\hat{G}):= \End_{\hat{G}}(\hat{G}^{\oplus n})
\] 
where $\hat{G}^{\oplus n} := \hat{G}\otimes_{\FF_1}\widehat{[n]}$. Formally, elements of $I_n(\hat{G})$ can be thought of as matrices with coefficients in $\hat{G}$ with the property that each row and column contains at most one non-zero entry. Matrices with this property are said to be \emph{submonomial}. Multiplication in $I_n(\hat{G})$ is then given by ordinary matrix multiplication.\footnote{This an abuse of notation, of course, but a harmless one.} Note that we abbreviate $I_n = I_n(\hat{1})$ throughout. 
\end{mydef}  

\begin{mydef}\label{d.D_nG} 
$D_n(G)$ denotes the subgroup of invertible diagonal matrices in $I_n(\hat{G})$ (note the lack of a caret over $G$). In other words, $A = [a_{ij}] \in D_n(G)$ if and only if $a_{ij} = 0$ for $i\neq j$ and $a_{ii} \in G$ for all $1 \le i \le n$.
\end{mydef}

\begin{rmk} 
Note that $I_n$ is isomorphic to the usual inverse symmetric monoid considered in semigroup theory \cite{CliffPrest1961,Steinberg2016}.
\end{rmk}

\begin{mydef} 
For $g \in G$, let $gE_{ij}$ denote the standard matrix unit with a $g$ in the $ij^{th}$ entry and zeroes elsewhere. For $S \subseteq [n]$, define $E_S := \sum_{i \in S}{E_{ii}}$. For an injective function $f : S \rightarrow [n]$ and a function $c : f(S) \rightarrow G$, define  
\[
\displaystyle M_{S,f,c}:= \sum_{i \in S}{c(f(i))E_{f(i)i}}. 
\]
 If $S$ is understood from context, we abbreviate $M_{f,c} = M_{S,f,c}$. Note that $E_S$, $M_{S,f,c} \in I_n(\hat{G})$ for all choices of $S$, $f$, and $c$. In fact, every element of $I_n(\hat{G})$ is equal to $M_{S,f,c}$ for unique choices of $S$, $f$ and $c$. In the special case $G = \{1\}$, $c$ is necessarily a constant map, so we abbreviate $M_{S,f,c} = M_{S,f}$.
\end{mydef}   

\begin{rmk}\label{r.composition}
Suppose that $f : S \rightarrow [n]$ and $g : f(S) \rightarrow [n]$ are injections. Then it is easy to check that 
\begin{equation}
\displaystyle M_{f(S),g,d}M_{S,f,c} = M_{S,g\circ f, d\times (c\circ g^{-1})}. 
\end{equation}
 In particular, the element $A = M_{f(S),f^{-1}, 1/(c\circ f)}$ satisfies $AM_{S,f,c} = E_S$ and $M_{S,f,c}A = E_{f(S)}$. 
\end{rmk}     

\begin{mydef}\label{d.rank} 
The rank of $M_{S,f,c}$, denoted $\operatorname{rank} (M_{S,f,c})$, is by definition $|S|$. Since $f$ is assumed to be injective, this coincides with the usual rank of a matrix. We let $I_n(\hat{G})^{\le d}$ denote the set of all $A \in I_n(\hat{G})$ with $\operatorname{rank}(A) \le d$.
\end{mydef}

For the purposes of this paper, we desire proto-abelian and finitary categories. Furthermore, the categories we consider will arise as categories of functors into $\Vect_{\hat{G}}$, for some $G$. It should be unsurprising that $\Vect_{\hat{G}}$ is proto-abelian and finitary, that $\Fun(M,\Vect_{\hat{G}})$ remains so under suitable restrictions on $M$ and that in this case, versions of the Jordan-H\"older and Krull-Schmidt Theorems also hold in $\Fun(M,\Vect_{\hat{G}})$. Checking this is similar to the case of quiver representations over $\FF_1$, so some proofs have been omitted. In spite of its naturality, our approach contrasts with others previously explored in the literature, see the following comment for more details.

%Several alternate frameworks for $\FF_1$-representations of monoids have been considered in the literature. In spite of similarities to the categories considered in the present work, in each case there are important differences. See the following remark for more details.

\begin{rmk} Similar notions of non-additive representation theory have been discussed at various points of time. However, there are significant departures between these previous approaches and the one taken in the present article. For the convenience of the reader, some of these approaches are compared below: 
\begin{enumerate} 
\item If $A$ is a monoid, the category $C_A$ of finite $A$-acts is considered in \cite{szczesny2014hall}. In contrast to the categories we consider in the present article, the elements of $A$ are not required to act via $\FF_1$-linear maps, nor are morphisms between acts required to be $\FF_1$-linear. The same article also considers the subcategory $C_A^N$ with the same objects as $C_A$, but whose morphisms are required to be $\FF_1$-linear. Again, this category contains more objects than the categories we consider, as we require elements of $A$ to act via $\FF_1$-linear maps. In subsequent articles from Szczesny, $C_A$ is written as $\operatorname{Mod}(A)_{\FF_1}$ and is referred to as the category of representations of $A$ over $\FF_1$ \cite{BeersSzcz2019}. Both of the categories $C_A$ and $C_A^N$ are proto-exact. The category of semigroup representations over a semiring discussed in \cite{IzhRhoStein2011} is constructed in a similar fashion.
\item $\hat{G}$-linear vector spaces have been implicitly considered in the semigroup literature previously, see for instance \cite{Tully1964}. Note that the representations considered in \cite{Tully1964} are either required to be row-monomial or column-monomial, but not both (which is equivalent to what this source terms $\hat{G}$-linearity), and that this mirrors other treatments of semigroup representations with coefficients in a pointed group or semigroup \cite{CliffPrest1961, IzhRhoStein2011}. The closest exception that the author is aware of are the Brandt semigroups, whose axioms imply that they act on themselves via $\FF_1$-linear maps \cite{CliffPrest1961}. However, the monoids considered in this article are generally not Brandt (cf. Example \ref{e.BrandtCE}). It is unclear to the author to what extent the proto-abelian structure of $\Vect_{\hat{G}}$ has appeared in the literature for $G\neq \langle 1\rangle$.  
\item The categories $\Rep (M, \hat{G})$ considered in the present article are more similar to the categories of quiver representations over $\FF_1$ studied in \cite{szczesny2011representations, jun2020quiver, JunSis2023, preprint}. In contrast to the prior literature on $\FF_1$-representations of monoids, $\FF_1$-representations of quivers are defined to be categories of functors into $\Vect_{\FF_1}$, with natural transformations as morphisms. In particular, these categories are automatically proto-abelian. 
\end{enumerate}
\end{rmk}

%\textcolor{red}{
%\begin{pro} 
%Let $\C$ be a proto-exact category, and $M$ a finite category. Then the category $\Fun(M,\C)$ of all functors $F : M \rightarrow \C$ with natural transformations as morphisms is proto-exact. If $\C$ is proto-abelian, then so is $\Fun(M,\C)$. If $\C$ is finitary, then so is $\Fun(M,\C)$.
%\end{pro}}   

We now verify that $\Vect_{\hat{G}}$ is proto-abelian. We include the details for the convenience of the reader.

\begin{construction}\label{c.Gpushout}
Consider a diagram  
\begin{center} 
\begin{tikzcd} 
A \arrow[hook]{r}{i} \arrow[swap, two heads]{d}{p} & B  \\ 
C  &
\end{tikzcd} 
\end{center} 
in $\Vect_{\hat{G}}$, with $i$ an injection and $p$ a surjection. Let $D = B\oplus C/\sim$, where $\sim$ is the smallest equivalence relation satisfying $i(a)\sim p(a)$ for all $a \in A$. We can explicitly describe this equivalence relation as follows: if $b \not\in i(A)$ set $[b] = \{b\}$. Otherwise suppose $b = i(a)$ for some $a \in A$. If $p(a) \neq 0$ then $a \neq 0$ and so $i(a) \neq 0$. By Remark \ref{r.Gidem}, $p(a') \neq p(a)$ for all other $a' \neq a$, and similarly $i(a') \neq i(a)$. Thus, we may set $[p(a)] = \{ p(a), i(a)\}$. If $p(a) = 0$, then $i(a) \sim p(a) = 0$ and so $i(a) \in [0]$. Therefore, we have $[0] = \{ i(a) \mid p(a) = 0\}$. Note that this completely determines $\sim$, as any $c \in C$ is of the form $c = p(a)$ for some $a \in A$.   

For all $g \in G$ and $a \in A$, we have $gi(a) = i(ga) \sim p(ga) = gp(a)$. Thus, $D$ is naturally a $G$-act via $g[d] = [gd]$, where $g \in G$ and $[d]$ denotes the equivalence class of $d \in B\oplus C$. Consider the function $p_B : B \rightarrow D$ satisfying $p_B(b) = [b]$ for all $b \in B$. Then $p_B$ is $G$-equivariant and surjective: indeed, any $c \in C$ is of the form $c=p(a)$ for some $a \in A$, and thus $[c] = [p(a)] = [i(a)] = p_B(i(a))$. If $b \not\in i(A)$, then $[b] = \{b\}$ and thus $G[b] = Gb$ is free. For all $a \in A\setminus\{0\}$, if $[i(a)] \neq [0]$ then $[i(a)] = \{i(a),p(a)\}$ and thus $[i(a)] = g[i(a)] = [gi(a)]$ for some $g \in G$ implies $gi(a) \in \{ i(a), p(a)\}$ and so $g = 1$. Thus, $D$ is a $\hat{G}$-vector space and $p_B$ is a $\hat{G}$-linear map.

 We can also define a function $i_C : C \rightarrow D$ via $i_C(c) = [c]$ for all $c \in C$. This is easily seen to be a $\hat{G}$-linear injection. We can now extend our diagram to a commutative square  
\begin{center} 
\begin{tikzcd} 
A \arrow[hook]{r}{i} \arrow[swap, two heads]{d}{p} & B \arrow[two heads]{d}{p_B}  \\ 
C \arrow[swap, hook]{r}{i_C}  & D \\
\end{tikzcd} 
\end{center} 

$D$ is the pushout of $C\twoheadleftarrow A \hookrightarrow B$.
\end{construction}  

\begin{construction}\label{c.Gpullback}
Consider a diagram in $\Vect_{\hat{G}}$ of the form  
\begin{center}
\begin{tikzcd} 
 & B \arrow[two heads]{d}{p'} \\ 
C \arrow[swap, hook]{r}{i'} & D
\end{tikzcd} 
\end{center} 
with $i'$ an injection and $p'$ a surjection. Set $A$ be the subset of $B\times C$ defined by  
\[
A = \{ (b,c) \mid p'(b) = i'(c)\}. 
\] 
Then $p'(b) = i'(c)$ implies $p'(gb) = gp'(b) = gi'(c) = i'(gc)$ for all $g \in G$, and so $A$ is a $G$-act under $g(b,c) = (gb,gc)$. Note that if $(0,c) \in A$ then $0 = p'(0) = i'(c)$ and so $c = 0$ by injectivity. More generally, for any $b \in B$, Remark \ref{r.Gidem} implies that there is at most one $c \in C$ satisfying $(b,c) \in A$. 

We now verify that $A$ is a $\hat{G}$-vector space. Clearly, $(0,0) \in A$ by $\hat{G}$-linearity and $\{(0,0)\}$ is an orbit.  If $g(b,c) = (b,c)$ for any $(b,c) \neq (0,0)$ in $A$ then either $b \neq 0$ or $c \neq 0$. If $b \neq 0$ then $gb = b$ and so $g = 1$ by freeness. Otherwise $b = 0$, which by the previous observation forces $c = 0$, a contradiction. In any case we conclude that all non-zero orbits of $A$ are free, and so $A$ is a $\hat{G}$-vector space. Projection onto the first and second components yield $G$-equivariant functions $i'_B: A \rightarrow B$ and $p'_C : A \rightarrow C$. The map $i'_B$ is injective by the remarks above, and $p'_C$ is surjective by the following argument: for any $c \in C$, $i'(c) = p'(b)$ for some $b \in B$ by the surjectivity of $p'$, and so $(b,c) \in A$. Furthermore, for a fixed $c \in C\setminus\{0\}$, $(b_1,c), (b_2,c) \in A$ implies $p'(b_1) = i'(c) = p'(b_2) \neq 0$ and so $b_1 = b_2$ by Remark \ref{r.Gidem}. It follows that $p'_C$ and $i'_B$ are $\hat{G}$-linear maps and we have a commutative square 
\begin{center} 
\begin{tikzcd} 
A \arrow[hook]{r}{i'_B} \arrow[swap, two heads]{d}{p'_C} & B \arrow[two heads]{d}{p'}  \\ 
C \arrow[swap, hook]{r}{i'}  & D \\
\end{tikzcd}  
\end{center} 
$A$ is the pullback of $B \twoheadrightarrow D \hookleftarrow C$.
\end{construction}

%\begin{textcolor}{red}{ 
\begin{lem}\label{l.Gprotoab}
Let $G$ be a finite abelian group. Then $\Vect_{\hat{G}}$ is a finitary proto-abelian category.
\end{lem} 
\begin{proof} 
To show that $\Vect_{\hat{G}}$ is proto-abelian, all that remains to be verified is Axiom (3). Consider the pushout diagram 
\begin{center} 
\begin{tikzcd} 
A \arrow[hook]{r}{i} \arrow[swap, two heads]{d}{p} & B \arrow[two heads]{d}{p_B}  \\ 
C \arrow[swap, hook]{r}{i_C}  & D \\
\end{tikzcd}   
\end{center}
in Construction \ref{c.Gpushout}. Suppose that two maps $\alpha : X \rightarrow B$ and $\beta : X \rightarrow C$ are given which satisfy $p_B \circ \alpha = i_C \circ \beta$. Given $x \in X$, suppose first that $\beta(x) \neq 0$ in $C$. Then $\beta(x) = p(a_x)$ for a unique $a_x \in A$. In this case, set $h(x) = a_x$. Note that $[0] \neq [\alpha(x)] = [\beta(x)] = [p(a_x)] = [i(a_x)]$ implies $\alpha(x) = i(a_x) = [i\circ h](x)$ in this case. Otherwise, $\beta(x) = 0$. Thus, $[\alpha(x)] = [\beta(x)] = [0]$ and so $\alpha(x) = i(a_x)$ for a unique $a_x \in A$. Again, we set $h(x) = a_x$. This defines a function $h : X \rightarrow A$ making the diagram 
\begin{center} 
\begin{tikzcd}  
X  \arrow[bend left= 30]{drr}{\alpha} \arrow[swap, bend right = 30]{ddr}{\beta} \arrow[]{dr}{h} & & \\
& A \arrow[hook]{r}{i} \arrow[swap, two heads]{d}{p} & B \arrow[two heads]{d}{p_B}  \\ 
& C \arrow[swap, hook]{r}{i_C}  & D \\
\end{tikzcd}   
\end{center} 
commute. It is straightforward to check that $a_{gx} = ga_x$ and $a_x = a_y \neq 0$ implies $x = y$, so that $h$ is a $\hat{G}$-linear map. Clearly, $h$ is unique.

Conversely, consider the pullback diagram in Construction \ref{c.Gpullback}:  
\begin{center} 
\begin{tikzcd} 
A \arrow[hook]{r}{i'_B} \arrow[swap, two heads]{d}{p'_C} & B \arrow[two heads]{d}{p'}  \\ 
C \arrow[swap, hook]{r}{i'}  & D \\
\end{tikzcd}  
\end{center}  
Suppose that $\alpha : B \rightarrow Y$ and $\beta : C \rightarrow Y$ are $\hat{G}$-linear maps satisfying $\alpha \circ i'_B = \beta \circ p'_C$. Given $y \neq 0$ in $D$, $y = p'(b_y)$ for a unique non-zero $b_y \in B$. Setting $b_0 = 0$, the map $y \mapsto \alpha(b_y)$ defines a function $h : D \rightarrow Y$. Again, it is straightforward to verify that $b_{gy} = gb_y$, and $h(x) = h(y) \neq 0$ implies $b_x = b_y$, and hence $x = p'(b_x) = p'(b_y) = y$. Thus $h$ is $\hat{G}$-linear. To verify that $\alpha = h\circ p'$, first note that this equality holds for all points of the form $b_y$, $y \in D$. Any remaining point $b \in B$ must satisfy $p'(b) = 0$, and hence $(b,0) \in A$. But then $\alpha(b) = i'(0) = 0$ and so $\alpha(b) = 0 = h(0) = h(p'(b))$ in this case as well. To verify $h\circ i' = \beta$, simply note that if $c \neq 0$ in $C$, then $i'(c) = p'\left(b_{i'(c)} \right)$ and so $(b_{i'(c)}, c) \in A$. Thus $\beta(c) = \alpha(b_{i'(c)}) = h(i'(c))$. It follows that we have a commutative diagram 
\begin{center} 
\begin{tikzcd} 
A \arrow[hook]{r}{i'_B} \arrow[swap, two heads]{d}{p'_C} & B \arrow[two heads]{d}{p'} \arrow[bend left = 30]{ddr}{\alpha} & \\ 
C \arrow[swap, hook]{r}{i'} \arrow[swap, bend right = 30]{drr}{\beta}  & D \arrow[]{dr}{h} & \\ 
& & Y \\
\end{tikzcd}  
\end{center}   
Clearly $h$ is the unique map making this diagram commute. We conclude that Axiom (3) holds and $\Vect_{\hat{G}}$ is proto-abelian. To verify it is finitary, fix any two objects $X$ and $Y$ of $\Vect_{\hat{G}}$. Then $X$ and $Y$ are finite sets, and in particular there are finitely-many $\hat{G}$-linear maps between them. Thus, $\Hom_{\hat{G}}(X,Y)$ is finite. Finally, if $0 \rightarrow Y \xrightarrow[]{f} E \xrightarrow[]{g} X \rightarrow 0$ is an admissible short exact sequence, then $f(Y)$ is a $\hat{G}$-vector subspace of $E$ isomorphic to $Y$ and $(E\setminus f(Y)) \cup \{ 0_E\}$ is a $\hat{G}$-vector space isomorphic to $E/f(Y) \cong X$. In other words,  $E \cong X\oplus Y$. From this, it easily follows that $\operatorname{Ext}_{\hat{G}}(X,Y)$ is a singleton set, and hence $\Vect_{\hat{G}}$ is finitary.
\end{proof}
%}

Recall that a category $M$ is \emph{finite} if it has finitely-many objects and finitely-many morphisms.

\begin{mythm}\label{p.GFunprotoab}
Let $G$ be a finite abelian group, and $M$ a category. Then the category $\Fun(M,\Vect_{\hat{G}})$ of all functors $F : M \rightarrow \Vect_{\hat{G}}$ with natural transformations as morphisms is proto-abelian. If $M$ is finite, then $\Fun(M,\Vect_{\hat{G}})$ is finitary. 
\end{mythm} 
\begin{proof}   
$\Vect_{\hat{G}}$ is proto-abelian by Lemma \ref{l.Gprotoab}, and therefore $\Fun(M,\Vect_{\hat{G}})$ is proto-abelian by Example 2.4.4 of \cite{dyckerhoff2012higher}.

Now suppose that $M$ is finite. Since there is an injective map  
\[
\displaystyle \Hom (A,B) \hookrightarrow \prod_{X \in \Obj(M)}{\Hom_{\hat{G}}(A(X),B(X))} 
\]
and $\Obj(M)$ is a finite set, we have that $\Hom(A,B)$ is a finite set for all $A,B \in \Fun(M,\Vect_{\hat{G}})$. If $\E: 0 \rightarrow B \rightarrow E \rightarrow A \rightarrow 0$ is an element of $\operatorname{Ext}^1(A,B)$, then note that $E(X) \cong A(X) \oplus B(X)$ for all objects $X$ of $M$. If $M$ has finitely-many morphisms then there are only finitely-many such $E$, up to isomorphism. Finally, there are obviously finitely-many $\E$ for any triple $(B,E,A)$ when $M$ has finitely-many objects, and so $\operatorname{Ext}^1(A,B)$ is a finite set. \end{proof}  

\begin{rmk} 
The elementary constructions in $\Vect_{\FF_1}$ described above work analogously in $\Vect_{\hat{G}}$. Specifically:
\begin{enumerate} 
\item A subspace $W \le V$ in $\Vect_{\hat{G}}$ is a union of $G$-orbits that contains the trivial orbit $\{0_V\}$. The quotient space $V/W$ is still defined as $(V\setminus W)\cup \{0_V\}$. 
\item The dimension of $V$ is now the number of non-trivial $G$-orbits $\dim_{\hat{G}}(V) = |\mcO_G(V)|-1$. A basis for $V$ is a selection $\{v_1,\ldots , v_n\}$ of representatives for the non-trivial $G$-orbits of $V$, where $n = \dim_{\hat{G}}(V)$.
\item Kernels and cokernels of morphisms $f : V \rightarrow W$ are defined as $\ker(f) = f^{-1}(0_W)$ and $\operatorname{coker}(f) = W/f(V)$. 
\item There is a symmetrical monoidal structure $\oplus$ defined via $V\oplus W = V\sqcup W/\{0_V,0_W\}$. 
\item Every non-zero object in $\Vect_{\hat{G}}$ is isomorphic to $\hat{G}\otimes_{\FF_1}\widehat{[n]}$ for a suitable positive integer $n$, and 
\[ 
\hat{G}\otimes_{\FF_1}\widehat{[n]} \cong (\hat{G}\otimes_{\FF_1}\widehat{[1]})^{\oplus n}
\]
\end{enumerate}
Subobjects, quotients, kernels, cokernels, and direct sums are defined pointwise in $\Fun(M,\Vect_{\hat{G}})$. Provided that $M$ is a finite category, the \emph{dimension} of $F \in \Fun(M,\Vect_{\hat{G}})$ is defined to be 
\[
\dim_{\hat{G}}(F) = \sum_{X \in \Obj(M)}{\dim_{\hat{G}}(F(X))} 
\] 
The \emph{dimension vector} of $F$ is the function $\underline{\dim}(F) : \Obj(M) \rightarrow \NN$ defined by $\underline{\dim}(F)(X) = \dim_{\hat{G}}(F(X))$.
\end{rmk}

We will also need the following results, whose proofs are analogous to those found in \cite{szczesny2011representations}.  

\begin{mydef} 
Let $M$ be a finite category. An object $V \in \Fun(M,\Vect_{\hat{G}})$ is \emph{simple} if $V \neq 0$ and $W \le V$ implies $W = 0$ or $W = V$. $V$ is \emph{semisimple} if it is isomorphic to a direct sum of simple objects. Finally, $V$ is \emph{indecomposable} if $V \neq 0$ and $V = V_1\oplus V_2$ implies that $V_1 = 0$ or $V_2 = 0$.
\end{mydef}   

\begin{mydef} 
Let $M$ be a finite category and $V \in \Fun(M,\Vect_{\hat{G}})$. A natural transformation $N : V \rightarrow V$ is said to be \emph{nilpotent} if there exists an $m \in \mathbb{N}$ such that  
\[
N^m = \underbrace{N\circ N\circ \cdots \circ N}_{m-\text{times}} 
\]
 is the zero transformation.
\end{mydef}

\begin{lem}\label{l.indecompmorph}
Let $M$ be a finite category, and suppose that $V \in \Fun(M,\Vect_{\hat{G}})$ is indecomposable. Then any endomorphism of $V$ is either nilpotent or an isomorphism.
\end{lem} 

\begin{proof} 
Let $F : V \rightarrow V$ be a natural transformation. Define $V^{\nil} \le V$ by the condition $V^{\nil}(X) = \{ m \in V(X) \mid F_X^k(m) = 0\text{ for some $k \in \mathbb{N}$}\}$ and $V^{\iso} \le V$ by the condition $V^{\iso}(X) = \{ m \in V(X) \mid F_X^k(m) \neq 0\text{ for all $k \in \mathbb{N}$} \}$. Then $V^{\nil}$ and $V^{\iso}$ are sub-objects of $V$ and $V = V^{\nil} \oplus V^{\iso}$. Since $V$ is indecomposable, either $V = V^{\nil}$ or $V = V^{\iso}$. If $V = V^{\nil}$ then $F$ is nilpotent, since $\bigcup_{X \in \Obj(M)}{V(X)}$ is a finite set. Otherwise $V = V^{\iso}$, in which case $F_X$ is necessarily a bijection for all $X \in \Obj(M)$, and hence $F$ is an isomorphism.
\end{proof}

\begin{mythm}[Krull-Schmidt Theorem]
Let $M$ be a finite category. Then every object in  \\
$\Fun(M,\Vect_{\hat{G}})$ is isomorphic to a direct sum of indecomposable objects, and the indecomposable summands are unique up to isomorphism and permutation of factors.
\end{mythm}  

\begin{proof} 
The same proof as in Theorem 4 in \cite{szczesny2011representations} works here, substituting Lemma \ref{l.indecompmorph} for Szczesny's Lemma 4.2.
\end{proof}

\begin{mythm}[Jordan-H\"older Theorem] 
Suppose that $M$ is a finite category and $V \in \Fun(M,\Vect_{\hat{G}})$. Let $0 = V_0 \le V_1 \le \ldots \le V_n = V$ and $0 = V_0' \le V_1' \le \ldots \le V_m' = V$ be two filtrations of $V$ whose successive quotients $A_j = V_j/V_{j-1}$ and $B_j = V_j'/V_{j-1}'$ are simple for all $j$. Then $m=n$ and there is a permutation $\sigma$ of $\{1,\ldots , n\}$ such that $A_j \cong B_{\sigma(j)}$ for all $j = 1,\ldots , n$.
\end{mythm} 

\begin{proof} 
The same proof as in Theorem 3 in \cite{szczesny2011representations} works here.
\end{proof}

Finally, many of the categories considered in this article are semisimple, in the sense defined below. We prove some elementary properties of such categories.

\begin{mydef}\label{d.ss}
Let $\C$ be a full proto-abelian subcategory of $\Fun(M,\Vect_{\hat{G}})$ that is closed under direct sums, and let $A,B \in \C$. A short exact sequence $0 \rightarrow A \rightarrow E \rightarrow B \rightarrow 0$ in $\C$ is said to be \emph{split} if it is equivalent to $0 \rightarrow A \rightarrow A\oplus B \rightarrow B \rightarrow 0$ (with the obvious inclusion and projection maps). $\C$ is said to be \emph{split-semisimple} if every short exact sequence in $\C$ splits, and $\emph{semisimple}$ if every non-zero object in $\C$ is isomorphic to a direct sum of simple objects. An object $V \in \C$ is said to be \emph{completely reducible} if whenever $S \in \C$ and $S \le V$, there exists a $T \in \C$ with $T \le V$ and $V = S\oplus T$. 
\end{mydef}     

%\begin{lem} 
%Let $M$ be a finite category. Then for any $V \in \Fun(M,\Vect(\hat{G}))$ with $S \le T \le V$, we have an isomorphism 
%\[ 
%V/T \cong (V/S)/(T/S)
%\]
%\end{lem} 

The following proposition will allow us to prove the semisimplicity of the categories discussed in this article.

\begin{pro}\label{p.verifyss}
Let $M$ be a finite category, and suppose that $\C$ is a full subcategory of $\Fun(M,\Vect_{\hat{G}})$ with the property that for all short exact sequences  
\[
0 \rightarrow A \rightarrow E \rightarrow B \rightarrow 0 
\]
in $\Fun(M,\Vect_{\hat{G}})$, $E \in \C$ if and only if $A,B \in \C$. Then the following are equivalent: 
\begin{enumerate}
 \item [(a)] $\C$ is split-semisimple.
 \item [(b)] Every short exact sequence $0 \rightarrow A \rightarrow E \rightarrow B \rightarrow 0$ in $\C$ with $A$ simple splits. 
 \item [(c)] Every object of $\C$ is completely reducible. 
 \item [(d)] $\C$ is semisimple.
 \end{enumerate}
\end{pro} 

\begin{proof} 
To begin, note that any such category $\C$ is a finitary proto-abelian category using the proto-abelian structure of $\Fun(M,\Vect_{\hat{G}})$. Furthermore it is a full subcategory of $\Fun(M,\Vect_{\hat{G}})$ by definition, and closed under taking direct sums, because any two $A, B \in \C$ have the split short exact sequence 
\[ 
0 \rightarrow A \rightarrow A\oplus B \rightarrow B \rightarrow 0
\]
 in $\Fun(M,\Vect_{\hat{G}})$. Thus, $\C$ satisfies the assumptions of Definition \ref{d.ss}. Clearly (a)$\Rightarrow$(b), and (c)$\Leftrightarrow$(d) because every object in $\C$ has finite dimension. To show (c)$\Rightarrow$(a), note that if $0 \rightarrow A \rightarrow E \rightarrow B \rightarrow 0$ is a short exact sequence in $\C$ and $E$ is completely reducible, then $E \cong A \oplus B$ by the Krull-Schmidt Theorem and the short exact sequence splits. Finally, we prove (b)$\Rightarrow$(d). Any non-zero $V \in \C$ contains a simple subobject $S \le V$, which induces a short exact sequence $0 \rightarrow S \rightarrow V \rightarrow V/S \rightarrow 0$. By hypothesis, this short exact sequence splits and $V \cong (V/S)\oplus S$. Note that $\dim(V/S) < \dim(V)$, and so the result now follows from induction on dimension.
\end{proof} 

\begin{rmk} 
The notion of a split proto-exact category is studied in \cite{Eberhardt2022AKGW1, Eberhardt2022AKGW2}, which in the proto-abelian case coincides with what we have called split-semisimple. We note that in general, a split proto-exact category may contain objects which are not direct sums of their simple subobjects. For instance, let $M$ be a finite monoid with absorbing element. Then Lemma 2.4 of \cite{Eberhardt2022AKGW1} implies that $A$-$\operatorname{proj}$ (i.e. the full subcategory of finite, pointed, projective left $M$-actions) is proto-exact and split-semisimple. However, it is easy to see that if $M$ is the path monoid of the quiver $Q =a \xrightarrow[]{\alpha} b$ (see Example \ref{e.pathmonoids}), then $Mb = \{0, b, \alpha\} \in M$-$\operatorname{proj}$ contains the projective sub-action $\{0,\alpha\} \cong Ma$, which has no complement in $Mb$. 
\end{rmk}

%%%%%%%%%%%% (REPS OF G-LINEAR MONOIDS )%%%%%%%%%%%%%%%%%%%%

\section{Representations of $\hat{G}$-Linear Monoids}\label{s.GLin}

We now let $M$ be a finite monoid, considered as a finite category with a single element. We wish to define a proto-abelian category $\Rep(M,\hat{G})$ which exhibits nice algebraic properties. To do this, it is wise to restrict both the class of monoids $M$ under consideration, as well as the functors $M \rightarrow \Vect_{\hat{G}}$ themselves. 

\begin{mydef} 
Let $G$ be a finite abelian group, and $M$ a finite monoid with absorbing element $0$. We say that $M$ is a \emph{$\hat{G}$-linear monoid} if $G$ is a subgroup of units of $M$ that commutes pointwise with the elements of $M$, and the action of $G$ on $M\setminus\{0_M\}$ via translation is free.
\end{mydef}   

Thus, a $\hat{G}$-linear monoid is in particular an object of $\Vect_{\hat{G}}$ in a natural way. The \emph{dimension} of a $\hat{G}$-linear monoid $M$ is its dimension as a $\hat{G}$-vector space. 

\begin{rmk} 
Let $M$ be a $\hat{G}$-linear monoid. Although $M$ is an object in $\Vect_{\hat{G}}$, it is not generally a monoid object in $\Vect_{\hat{G}}$: for instance, left and right translation by fixed elements of $M$ are not generally $\hat{G}$-linear maps $M\rightarrow M$. Our definition is chosen so that monoids such as $I_n(\hat{G})$ might be included within it.
\end{rmk}

\begin{mydef} 
A monoid $M$ is said to be an \emph{inverse monoid} if for all $x \in M$, there exists a unique $x^* \in M$ satisfying $x^*xx^* = x^*$ and $xx^*x = x$. We call $x^*$ the \emph{$*$-inverse of $x$.} More generally, an element $x$ of a monoid $M$ (not necessarily inverse) is said to be \emph{regular} if and only if there exists a $y \in M$ satisfying $xyx = x$. 
\end{mydef}  

We will need the following elementary facts about $*$-inverses and inverse monoids, see \cite{CliffPrest1961} for proofs: 

\begin{pro}\label{p.BasicInverses}
Let $M$ be a finite monoid. Then $M$ is an inverse monoid if and only if each element of $M$ is regular and all idempotents of $M$ commute. Furthermore, for all $a, b \in M$ we have the following: 
\begin{enumerate} 
\item $a^{**} = a$.
\item $(ab)^* = b^*a^*$. 
\item If $u$ is a unit in $M$, then $u^* = u^{-1}$. 
\item $a^*a$ and $aa^*$ are idempotents.
\end{enumerate}
\end{pro} 

Of course, $I_n(\hat{G})$ is a $\hat{G}$-linear inverse monoid for any finite abelian group $G$. We note the following version of the Wagner-Preston Theorem, adapted to suit the needs of this article. The proof is similar to that given for inverse monoids in \cite[Theorem 1.20]{CliffPrest1961}, but is included for the convenience of the reader.

\begin{mythm}[Wagner-Preston Theorem for $\hat{G}$-Linear Monoids, cf. Theorem 1.20 \cite{CliffPrest1961}]\label{t.WP}
Let $G$ be a finite abelian group, with $M$ a finite inverse $\hat{G}$-linear monoid of dimension $n$. Then there exists an injective monoid homomorphism $\phi_M: M \rightarrow \End_{\hat{G}}(M) \cong I_n(\hat{G})$ that satisfies $\phi_M(x^*) = \phi_M(x)^*$ and $\phi_M(gx) = g\phi_M(x)$ for all $x \in M$ and $g \in G$. 
\end{mythm} 

\begin{proof}  
For all $a \in M$, define $\lambda_a : M \rightarrow M$ as 
\begin{center} 
$\lambda_a(x) = \begin{cases} 
ax& \text{ if $x \in a^*M = (a^*aM)$} \\ 
0 & \text{ else}
\end{cases}$ 
\end{center} 
Then by definition, $\lambda_a(0_M) = a0_M = 0_M$. If $g \in G$, then $x \in a^*M$ if and only if $gx \in a^*M$, because $g$ is invertible and central. Thus, $\lambda_a(gx) = g\lambda_a(x)$ for all $x \in M$. Furthermore, $\lambda_a(x) = \lambda_a(y) \neq 0$ implies that $ax = ay$ and $x,y \in a^*M$. Write $x = a^*z$ for some $z \in M$. Then $x = a^*z = (a^*aa^*)z = (a^*a)(a^*z) = (a^*a)x$, and similarly $y = (a^*a)y$. Therefore, $ax = ay$ implies $x = a^*ax = a^*ay = y$, and so $\lambda_a \in \End_{\hat{G}}(M)$ for all $a \in M$. It is clear that $\phi_M(1) = \lambda_1 = \id_M$ by the definition of $\lambda_1$. To prove that $\phi_M$ is a monoid homomorphism, suppose $a, b \in M$. Then on the one hand, 
\begin{center} 
 $[\lambda_a\circ \lambda_b](x) = \begin{cases} 
 abx & \text{ if $x \in b^*M$ and $bx \in a^*M$} \\ 
 0 & \text{ else} 
 \end{cases}$
\end{center} 
whereas 
\begin{center} 
$\lambda_{ab}(x) = \begin{cases} 
abx & \text{ if $x \in (ab)^*M$} \\ 
0 & \text{ else}
\end{cases}$
\end{center} 
Thus, we will have $\phi_M(ab) = \phi_M(a)\phi_M(b)$ if $x \in (ab)^*M$ if and only if $x \in b^*M$ and $bx \in a^*M$. If $x = (ab)^*z = b^*(a^*z)$, then immediately $x \in b^*M$. But also $a^*z = (a^*aa^*)z$ implies  
\begin{align*} 
bx & = b(b^*a^*z) \\ 
& = (bb^*)(a^*a)(a^*z) \\ 
& = (a^*a)(b^*b)(a^*z) \\
\end{align*} 
by Proposition \ref{p.BasicInverses}, which is clearly in $a^*M$. Conversely, suppose $x = b^*z$ and $bx = a^*w$ for some $z, w \in M$. Then 
\begin{align*} 
x & = b^*z \\ 
& = (b^*bb^*)z \\ 
& = b^*bx \\ 
& = b^*a^*w \\ 
& = (ab)^*w\\
\end{align*} 
from which we conclude $x \in (ab)^*M$. Thus, $\phi_M$ is a monoid homomorphism. But since $x \in a^*M$ if and only if $x \in (ga)^*M$ because $(ga)^* = a^*g^* = a^*g^{-1}$, it follows that $\lambda_{ga} = g\lambda_a$ for all $g \in G$, and thus $\phi_M(ga) = g\phi_M(a)$ as well. Furthermore, $a^*aa^* = a^*$ and $aa^*a = a$ imply $\phi_M(a^*)\phi_M(a)\phi_M(a^*) = \phi_M(a^*)$ and $\phi_M(a)\phi_M(a^*)\phi_M(a) = \phi_M(a)$. Since $\End_{\hat{G}}(M) \cong I_n(\hat{G})$ is an inverse monoid, we must have $\phi_M(a^*) = \phi_M(a)^*$ by the uniqueness of $*$-inverses. Finally, we verify that $\phi_M$ is injective. Suppose that $a, b \in M$ satisfy $\lambda_a = \lambda_b$. Since $b^*b \in b^*M$, we have 
\begin{align*} 
b & = bb^*b \\ 
& = \lambda_b(b^*b) \\ 
& = \lambda_a(b^*b) \\ 
& = ab^*b,
\end{align*} 
which in turn implies 
\begin{align*}  
b^* & = b^*bb^* \\  
 & = b^*(ab^*b)b^* \\ 
& = b^*a(b^*bb^*) \\ 
& = b^*ab^*
\end{align*} 
Similarly, we may conclude $a^* = a^*ba^*$, which in turn implies  
\begin{align*} 
a & = a^{**} \\ 
& = (a^*ba^*)^* \\ 
& = ab^*a,
\end{align*} 
which together imply $a = (b^*)^* = b$ by uniqueness of $*$-inverses. Thus, $\phi_M$ is injective, and the claim now follows.
\end{proof} 

The above theorem implies that in order to study $\hat{G}$-linear inverse monoids, it suffices to study submonoids of $I_n(\hat{G})$ that are closed under taking $*$-inverses. \\  

We now recall further notions from semigroup theory. 

\begin{mydef}\label{d.ReesFactor}
An \emph{ideal} of $M$ is a subset $I \subseteq M$ such that $aI \subseteq I$ and $Ia\subseteq I$ for all $a \in M$. Provided that $M$ has an absorbing element $0$, $I$ is said to be \emph{nilpotent} if $I^k = 0$ for suitably large $k$. For an ideal $I$ of $M$, we define an equivalence relation $\sim$ on $M$ by $x \sim y$ if and only if either $x = y$ or both $x$ and $y$ are in $I$. The equivalence class of $x \in M$ is denoted by $[x]$. The set of equivalence classes is denoted by $M/I$ and is called the \emph{Rees factor monoid} of $M$ by $I$. It is naturally a monoid with absorbing element, whose multiplication is given by 
\[ 
[a]\cdot [b] = \begin{cases} 
[ab] & \text{ if $a, b, ab \in M\setminus I$} \\ 
[0] & \text{ else} \\
\end{cases}
\]
\end{mydef}  

\begin{rmk}\label{r.ReesFactor} 
Let $I$ be an ideal of $M$. Note that if $M$ is finite or $\hat{G}$-linear, then so is the Rees factor monoid $M/I$. Note that the construction of Rees factor monoids also works for semigroups.
\end{rmk}

\begin{mydef}
For any $a \in M$, let $J(a) := MaM$ be the two-sided ideal of $M$ generated by $a$. Two elements $a, b \in M$ are said to be \emph{$\mathcal{J}$-equivalent} if and only if $J(a) = J(b)$. This is an equivalence relation on $M$, whose equivalence classes are referred to as \emph{$\mathcal{J}$-classes of $M$}. The $\mathcal{J}$-class of $a$ is denoted $\mathcal{J}_a$.
\end{mydef} 

Some elementary properties of the $\mathcal{J}$-equivalence follow.

\begin{pro}[Fact 2 \cite{GanMazStein2009}]\label{p.regular}
Let $M$ be a finite semigroup, and let $J$ be a $\mathcal{J}$-class of $M$. Then the following are equivalent: 
\begin{enumerate} 
\item $J$ contains an idempotent.
\item $J$ contains a regular element.
\item All elements of $J$ are regular. 
\item $J^2\cap J \neq \emptyset$.
\end{enumerate}
\end{pro} 

A $\mathcal{J}$-class satisfying any of the four equivalent conditions above is called a \emph{regular} $\mathcal{J}$-class, and its elements are called \emph{regular elements}. The set of all $\mathcal{J}$-classes of $M$ is a poset under the following order: $\mathcal{J}_a \le_{\mathcal{J}} \mathcal{J}_b$ if and only if $J(a) \subseteq J(b)$. \\ 

\begin{mydef} 
Let $a \in M$, and set $I(a)$ to be the collection of all $x \in J(a)$ which do not generate $J(a)$. Then $I(a)$ is an ideal of $J(a)$, and the Rees factor semigroup $P(a) = J(a)/I(a)$ is called a \emph{principal factor of $M$}. Clearly, the principal factors of $M$ only depend on its $\mathcal{J}$-classes. A principal factor of $M$ is said to be \emph{regular} if it corresponds to a regular $\mathcal{J}$-class, i.e. if and only if it is generated by an idempotent.
\end{mydef}

If $S$ is a semigroup with $0$, then $S$ is said to be \emph{null} if $S^2 = 0$. $S$ is said to be $0$-simple if and only if $S^2 \neq 0$ and the only two ideals of $S$ are $S$ and $\{0\}$.  

\begin{pro}[Lemma 2.39 \cite{CliffPrest1961}]\label{p.PFactClass}
Let $M$ be a finite monoid with absorbing element. Then the principle factors of $M$ are either null or $0$-simple.
\end{pro}  

The following lemma will be used in the next section: 

\begin{pro}\label{p.nilpotentcondition} 
Let $M$ be a finite monoid with absorbing element $0$, and let $I$ be a non-zero ideal of $M$. If $0$ is the only regular element of $I$, then $I$ is nilpotent.
\end{pro} 

\begin{proof} 
Note that $I$ is a finite union of $\mathcal{J}$-classes, say $I = \mathcal{J}_{a_1}\cup \cdots \cup \mathcal{J}_{a_k}$ for suitable $a_1,\ldots , a_k \in I$. Suppose by way of contradiction that $I$ is not nilpotent. First, we claim that $I^2 \subsetneq I$. Without loss of generality we may assume that $\mathcal{J}_{a_1}$ is maximal in $I$ with respect to the partial ordering $\le_{\mathcal{J}}$. By Proposition \ref{p.regular} no element of $J(a_1)^2$ generates $J(a_1)$, and hence $J(a_1)^2 \subsetneq J(a_1)$. If $I^2 = I$, then in particular there would exist $a, b \in I$ such that $a_1 = ab$. But $J(ab) \subseteq J(a)\cap J(b)$ implies $\mathcal{J}_{a_1} \le_{\mathcal{J}} \mathcal{J}_{a}$ and $\mathcal{J}_{a_1} \le_{\mathcal{J}} \mathcal{J}_b$, and hence $\mathcal{J}_{a_1} = \mathcal{J}_a = \mathcal{J}_b$ by maximality. In particular, $a, b \in J(a_1)$ and hence $a_1 = ab \in J(a_1)^2$, a contradiction. So $I^2 \subsetneq I$ as claimed. By assumption $I^2 \neq 0$, and thus we can repeat this argument with $I^2$ in place of $I$. We create a strictly decreasing chain of non-zero ideals 
\[ 
I \supsetneq I^2 \supsetneq I^4 \supsetneq \ldots \supsetneq I^{2^j} \supsetneq \ldots
\] 
contrary to the assumption that $M$ is finite. Thus, $I$ must be nilpotent.
\end{proof}

We now turn to the study of representations of $\hat{G}$-linear monoids. First, we need a suitable category of representations:

\begin{mydef} 
Suppose that $M$ is a $\hat{G}$-linear monoid. Define $\Rep(M,\hat{G})$ to be the category of functors $F : M \rightarrow \Vect_{\hat{G}}$ satisfying $F(0_M) = 0$ and $F(g) = $ scalar multiplication by $g$ for all $g \in G$. Note that this is compatible with the case $G = \{1\}$, as $F(1_M)$ is the identity map by functoriality. Objects of $\Rep(M,\hat{G})$ are called $\hat{G}$-linear representations of $M$.
\end{mydef}   

It is clear that $\Rep(M,\hat{G})$ is a full subcategory of $\Fun(M,\Vect_{\hat{G}})$, and that for any short exact sequence 
\[ 
0 \rightarrow B \rightarrow E \rightarrow A \rightarrow 0
\] 
in $\Fun(M,\Vect_{\hat{G}})$, we have $E \in \Rep(M,\hat{G})$ if and only if $A,B \in \Rep(M,\hat{G})$. From the observation at the start of the proof of Proposition \ref{p.verifyss}, $\Rep(M,\hat{G})$ is a finitary proto-abelian category that is closed under direct sums.

\begin{rmk} 
The requirement that elements of $G$ act as scalar multiples of the identity ensures that $\Rep(\hat{G},\hat{G})$ is semisimple, with a single simple object. The category $\Fun(\hat{G},\Vect_{\hat{G}})$ will in general contain one isomorphism class of simple object for each group homomorphism $G \rightarrow G$. 
\end{rmk}  

\begin{rmk} 
$\Rep(M,\hat{G})$ is closed under direct sums, sub-objects and quotient objects in $\Fun(M,\Vect_{\hat{G}})$. In particular, the Jordan-H\"older and Krull-Schmidt Theorems still hold in $\Rep(M,\hat{G})$.
\end{rmk}   

Let $J$ be a $\mathcal{J}$-class of $M$. Define $I_J = \{ x \in M \mid J\not\subseteq MxM\}$. Then $I_J$ is easily seen to be a two-sided ideal in $M$.  

\begin{mydef} 
Let $V \in \Rep(M,\hat{G})$. A regular $\mathcal{J}$-class $J$ is said to be the \emph{apex} of $V$ if and only if $\operatorname{Ann}_M(V) = I_J$. $J$ is the apex of $V$ if and only if it is the unique minimal $\mathcal{J}$-class (with respect to $\le_{\mathcal{J}}$) which does not annihilate $V$.
\end{mydef} 

We now develop a version of Clifford-Munn-Ponizovski\u i (CMP) Theory for $\hat{G}$-linear representations of monoids. To start, we verify some properties of simple $\hat{G}$-linear representations. The following proof, whose result is attributed to Munn and Ponizovski\u i, is essentially Theorem 5 in \cite{GanMazStein2009}. We reproduce it here for the convenience of the reader. 

\begin{pro}[Theorem 5 \cite{GanMazStein2009}]\label{p.apex} 
Let $V \in \Rep(M,\hat{G})$ be simple. Then $V$ has an apex. 
\end{pro} 

\begin{proof} 
Since $MV = V \neq 0$, there is a $\le_{\mathcal{J}}$-minimal $\mathcal{J}$-class $J$ such that $J \not\subseteq \operatorname{Ann}_M(V)$. Let $I = MJM$ be the two-sided ideal of $M$ generated by $J$. Note that $I$ is a union of $\mathcal{J}$-classes: indeed, if $x \in I$ and $MxM = MyM$, then $y \in MxM \subseteq I$. In particular, $I\setminus J$ is also a union of $\mathcal{J}$-classes. If $x \in I\setminus J$, then by definition of $I$ we have $x = \lambda j \rho$ for some $\lambda, \rho \in M$ and $j \in J$. But then $MxM \subseteq MjM$ implies $\mathcal{J}_x \le_{\mathcal{J}} J$. Since $\mathcal{J}_x \neq J$, this implies that $x \in \operatorname{Ann}_M(V)$, and hence $I\setminus J \subseteq \operatorname{Ann}_M(V)$. But then $IV = JV \neq 0$, and so $IV = V$ by simplicity. However, $I_JJ \subseteq I\setminus J \subseteq \operatorname{Ann}_M(V)$ by definition, so it follows that $I_JV = 0$. Furthermore, $xV = 0$ clearly necessitates $J \not\subseteq MxM$, and so we conclude $\operatorname{Ann}_M(V) = I_J$. If $J$ is not regular, then Proposition \ref{p.regular} implies $J^2 \subseteq I\setminus J$ and hence $V = JV = J^2V = 0$, a contradiction. Thus $J$ is regular, and so is an apex for $V$.
\end{proof}  

To define induction functors on categories of $\hat{G}$-linear representations, we make the following definition. 

\begin{mydef} 
Let $M$ be a finite $\hat{G}$-linear monoid. We say that $M$ is \emph{left (resp. right) inductive} if each regular principal factor of $M$ is a $\hat{G}$-linear representation of $M$ under left (resp. right) translation. 
\end{mydef}  

Note that there exist left-inductive monoids that are not right-inductive, and vice-versa (cf. Example \ref{e.LRCE}).

\begin{pro} 
Let $M$ be a finite, $\hat{G}$-linear, left inductive monoid, and let $e$ be a non-zero idempotent of $M$. If $P(e)$ is the associated regular principal factor of $M$, then $\mathcal{J}_e$ is an apex for $P(e)$ and the submodule $P(e)e$.
\end{pro} 

\begin{proof} 
We prove the claim for $P(e)$, the claim for $P(e)e$ follows from a similar argument. To begin, define $J = \mathcal{J}_e$, and suppose $x \in \operatorname{Ann}_M(P(e))$. Since $eP(e) \neq 0$, it follows that $J \not\le \mathcal{J}_x$, and hence $\operatorname{Ann}_M(P(e)) \subseteq I_J$. Conversely, suppose $x \in I_J$. Then $J(e) \not\subseteq J(x)$. If $j \in P(e)$, then $J(xj)$ is an ideal in $P(e)$. Since $e \neq 0$, $P(e)^2 \neq 0$ and thus $P(e)$ is $0$-simple. Therefore, either $J(xj) = 0$ or $J(xj) = P(e)$. But $J(xj) = P(e)$ implies that $J(e) = J(xj) \subseteq J(x)$, a contradiction. Thus $J(xj) = 0$ and in particular $xj = 0$. Since $j$ was an arbitrary element of $P(e)$ it follows that $x \in \operatorname{Ann}_M(P(e))$ and thus $\operatorname{Ann}_M(P(e)) = I_J$. 
\end{proof}  
 
Let $e \in M$ be an idempotent with $J = \mathcal{J}_e$ its associated regular $\mathcal{J}$-class. The set $eMe\cap J$ is a group under multiplication with identity element $e$. Up to isomorphism, this group only depends on $J$, and thus it is referred to as $G_J$ \cite[Corollaries 1.12, 1.16]{Steinberg2016}.

\begin{pro} 
Let $M$ be a finite, $\hat{G}$-linear, left inductive monoid, and let $e$ be a non-zero idempotent of $M$. Then there exists a positive integer $d$ such that $P(e)e \cong \hat{G}_J^{\oplus d}$ as right $\hat{G}$-linear $\hat{G}_J$-representations.
\end{pro}

\begin{proof}
It is clear that $P(e)e$ is a right $\hat{G}$-linear $\hat{G}_J$-representation. Suppose that $0 \neq j \in P(e)e$ and $\alpha \in G_J$ such that $j\alpha = j$. By simplicity, $MeM = MjM$ and hence $e = ljr$ for some $l, r \in M$. In particular, $lj \in eMe \cap J = G_J$, and so $j\alpha = j$ implies $lj\alpha = lj$. Thus, $\alpha = e$. It follows that $P(e)e\setminus 0$ is a free right $G_J$-set.  
\end{proof}

If $W \in \Rep(\hat{G}_J,\hat{G})$, we define $W{\uparrow_e}$ to be the tensor product
\[ 
W{\uparrow_e} = P(e)e\otimes_{\hat{G}_J}W
\] 
In other words, $W{\uparrow_e}$ is the pointed set $P(e)e\otimes_{\hat{G}}W/\sim$, where $\sim$ is the equivalence generated by $x\alpha \otimes w = x\otimes \alpha w$ for all $x \in P(e)e$, $\alpha \in \hat{G}_J$, and $w \in W$. 

\begin{lem}\label{l.ind} 
Let $M$ be a finite, $\hat{G}$-linear, left inductive monoid. Then for all $W \in \Rep(\hat{G}_J,\hat{G})$, $W{\uparrow_e}$ is a $\hat{G}$-linear representation of $M$ under the action $a(x\otimes w) = (ax)\otimes w$.
\end{lem} 

\begin{proof} 
As $(ga)\otimes w = (ga)\alpha \otimes \alpha^{-1}w$ for all $\alpha \in G_J$, $g\cdot (a\otimes w) = (ga)\otimes w$ is a $G$-action on $W{\uparrow_e}$. Furthermore, $a\otimes gw = ga \otimes w = a\otimes w \neq 0$ means $a = a\alpha$ and $gw = \alpha^{-1}w$ for some $\alpha \in G_J$. By the previous proposition, $\alpha = e$ and so $gw = w \neq 0$, from which it follows that $g = 1$. Thus, $W{\uparrow_e}$ is a $\hat{G}$-vector space. If $a(x\otimes w) = a(y\otimes z) \neq 0$, then there exists $\alpha \in G_J$ such that $ax = ay\alpha \neq 0$ and $ w = \alpha^{-1}z \neq 0$. Since $M$ is left inductive, $x = y\alpha$ and so $x\otimes w = y\alpha \otimes \alpha^{-1}z= y\otimes z$. Hence, $W{\uparrow_e} \in \Rep(M,\hat{G})$.
\end{proof}

\begin{mydef} 
Let $M$ be a finite, $\hat{G}$-linear, left inductive monoid. Let $e$ be a non-zero idempotent in $M$, and let $J = \mathcal{J}_e$ be the associated regular $\mathcal{J}$-class. Then \emph{induction along $e$} is the functor 
\[ 
(-){\uparrow_e}: \Rep(\hat{G}_J,\hat{G}) \rightarrow \Rep(M,\hat{G})
\] 
which carries $W \in \Rep(\hat{G}_J,\hat{G})$ to the representation $W{\uparrow_e}$ and $f : V \rightarrow W$ to the morphism  
\[
f{\uparrow_e} : V{\uparrow_e}\rightarrow W{\uparrow_e} 
\]
 defined by $[f{\uparrow_e}](x\otimes v) = x\otimes f(v)$.
\end{mydef} 

\begin{mydef} 
Let $M$ be a finite $\hat{G}$-linear monoid. Fix a non-zero idempotent $e$ in $M$, and let $J = \mathcal{J}_e$ be its associated regular $\mathcal{J}$-class. Then \emph{restriction along $e$} is the functor 
\[ 
(-){\downarrow_e}: \Rep (M,\hat{G}) \rightarrow \Rep(\hat{G}_J,\hat{G})
\] 
which carries $W \in \Rep(M,\hat{G})$ to the representation $W{\downarrow_e} = eW$ and a morphism $f : V\rightarrow W$ to the morphism  
\[
f{\downarrow_e} : V{\downarrow_e} \rightarrow W{\downarrow_e} 
\] 
defined by $f{\downarrow_e} = f\mid_{eV}$.
\end{mydef}  

The Clifford-Munn-Ponizovski\u i Theorem for $\hat{G}$-linear monoids follows below. Our proof is based on \cite[Theorem 7]{GanMazStein2009}, which works for coefficients taking values in an arbitrary commutative ring. Their result was later generalized to representations over commutative semirings in \cite{IzhRhoStein2011}. Note that our result is not a special case of the latter, as our category of representations is more restrictive than the usual notion.

\begin{mythm}[Clifford-Munn-Ponizovski\u i Theorem for $\hat{G}$-Linear Monoids, cf. Theorem 7 \cite{GanMazStein2009}]\label{t.CMP}
Let $M$ be a finite, $\hat{G}$-linear, left inductive monoid. Fix a non-zero idempotent $e \in M$, and let $J = \mathcal{J}_e$ be its associated regular $\mathcal{J}$-class. Then the following hold: 
\begin{enumerate}  
\item If $V \in \Rep(\hat{G}_J,\hat{G})$ is simple, then  
\[ 
N(V{\uparrow_e}) := \{x\otimes w \in V{\uparrow_e} \mid eM(x\otimes w) = 0 \} 
\] 
is the unique maximal submodule of $V{\uparrow_e}$. Furthermore, the simple representation $Q(V) = V{\uparrow_e}/N(V{\uparrow_e})$ has apex $J$ and $Q(V){\downarrow_e} \cong V$ in $\Rep(\hat{G}_J,\hat{G})$.  
\item If $V \in \Rep(M,\hat{G})$ is simple with apex $J$, then $W= V{\downarrow_e}$ is simple in $\Rep(\hat{G}_J,\hat{G})$. Furthermore,
\[ 
W{\uparrow_e}/N(W{\uparrow_e}) \cong V
\] 
in $\Rep(M,\hat{G})$.
\end{enumerate} 
In particular, there is a bijection between simple objects in $\Rep(M,\hat{G})$ with apex $J$ and simple objects in $\Rep(\hat{G}_J,\hat{G})$.
\end{mythm} 

\begin{proof} 
To prove (1), suppose that $V \in \Rep(\hat{G}_J,\hat{G})$ is simple. Then $V{\uparrow_e} \in \Rep(M,\hat{G})$ by Lemma \ref{l.ind}, and $N(V{\uparrow_e})$ is clearly a subrepresentation of $V{\uparrow_e}$. Note that $e\otimes V \neq 0$ in $V{\uparrow_e}$, and so $e^2 = e$ implies that $N(V{\uparrow_e})$ is proper. If $x\otimes w \not\in N(V{\uparrow_e})$, then $eM(x\otimes w) \neq 0$ and $e(Mx\otimes w) \subseteq e(V{\uparrow_e}) = e(P(e)e\otimes_{\hat{G}_J}V) = e\otimes V \cong V$ as a $\hat{G}_J$-representation, so $e(Mx\otimes w) = e(V{\uparrow_e})$ by the simplicity of $V$. In particular, $Me(Mx\otimes w) = Me(V{\uparrow_e})  = M(\hat{G}_J\otimes_{\hat{G}_J}V) = Me\otimes V  = P(e)e\otimes V = V{\uparrow_e}$. In other words, if $x\otimes w$ generates a proper subrepresentation of $V{\uparrow_e}$, then $x\otimes w \in N(V{\uparrow_e})$. Thus, $N(V{\uparrow_e})$ is the unique maximal subrepresentation of $V{\uparrow_e}$ and $Q(V) =V{\uparrow_e}/N(V{\uparrow_e})$ is a simple $\hat{G}$-linear representation of $M$. To prove that $J$ is the apex of $Q(V)$, first note that $eQ(V) \neq 0$ and so $JQ(V) \neq 0$ as well. Thus, $\operatorname{Ann}_M(Q(V)) \subseteq I_J$. Conversely, if $J \not\subseteq J(x)$ then $xP(e)e = 0$ and hence $x(V{\uparrow_e}) = 0$, which shows $I_J \subseteq \operatorname{Ann}_M(Q(V))$. Thus, $J$ is the apex of $Q(V)$. Furthermore, the elements of $Q(V){\downarrow_e} = eQ(V) = e[V{\uparrow_e}/N(V{\uparrow_e})]$ can be thought of as the union of the elements in $eM(x\otimes w)$, where $x\otimes w$ ranges through the tensors for which $eM(x\otimes w) \neq 0$. But by the argument above, such tensors satisfy $eM(x\otimes w) = e(V{\uparrow_e}) = e\otimes V $ and so $eQ(V) \cong V$ as $\hat{G}_J$-representations. Thus, claim (1) follows. 

To prove (2), first note that $eV \neq 0$ since $J$ is the apex of $V$. Pick any $v \in V$ such that $ev \neq 0$. Then $Mev \neq 0$ and so $Mev = V$ by simplicity. In particular, it follows that $eMev = eV$. But if $x \in eMe$ and $x \not\in J$, then $J(x) \subsetneq J(e)$ and hence $0 = xV = xeV = x(eMev)$. Therefore, $eV = eMev = (eMe \cap J)v\cup \{0\} = \hat{G}_Jv$, from which the simplicity of $W =V{\downarrow_e}$ follows. By (1), we know that $Q(W) = W{\uparrow_e}/N(W{\uparrow_e})$ is a simple $\hat{G}$-linear representation of $M$ with apex $J$. But the map $Q(W) \rightarrow V$ defined by $x\otimes v \mapsto xv$ is clearly a non-zero morphism in $\Rep(M,\hat{G})$, and so $Q(W) \cong V$ by simplicity.
\end{proof}   

We now end this section with some examples. 

\begin{myeg}[Path Monoids]\label{e.pathmonoids}
Let $Q$ be a finite acyclic quiver (i.e. $Q$ does not contain oriented cycles of positive length). Then $\hat{G}Q$ denotes the $\hat{G}$-linear path monoid of $Q$: the non-trivial $G$-orbits of $\hat{G}Q$ are in bijective correspondence with the oriented paths in $Q$, with multiplication defined via concatenation of paths: for paths $p_1$ and $p_2$ in $Q$ and $g_1, g_2 \in G$, $(g_1p_1)(g_2p_2) = g_1g_2(p_1p_2)$ if $t(p_1) = s(p_2)$, and $(g_1p_1)(g_2p_2) = 0$ otherwise. The non-trivial idempotents of $\hat{G}Q$ are in bijective correspondence with vertices of $Q$ and the unit $1$. For any vertex $v$ of $Q$, $P(v)$ is the one-dimensional simple $\hat{G}$-linear representation afforded by $v$ and $P(1)$ is the $\hat{G}$-linear representation afforded by the units of of $\hat{G}Q$. In particular, $\hat{G}Q$ is always left and right inductive. Quotients of $\hat{G}$-linear path monoids by relations may or may not remain inductive.
\end{myeg} 

\begin{myeg}[Left Inductive $\not\Rightarrow$ Brandt Principal Factors]\label{e.BrandtCE} 
Consider the quiver 
\[ Q = 
\begin{tikzcd} 
e \arrow[r, swap, bend right = 30, "\beta"] & \arrow[l, swap, bend right = 30, "\alpha"] f
\end{tikzcd} 
\]  
subject to the (non-admissible) relations $\alpha \beta = f$, $\beta \alpha = 0$. Then the resulting $\FF_1$-linear path monoid with relations is $M = \{ 0, e, f, \alpha, \beta, 1 \}$. It is easy to check that $P(e) = \{0, e\}$ is the $1$-dimensional simple representation of $Q$ at $e$, which is $\FF_1$-linear. On the other hand, $J(f) = \{0, f, \alpha, \beta\}$ and $\mathcal{J}_f = \{f, \alpha, \beta\}$. It is straightforward to check that left translation by any element of $M$ yields an $\FF_1$-linear map on $P(f) = J(f)/\{0\}$. The principal factor $P(1)$ is also $\FF_1$-linear via left translation, as in the previous example. In particular, $M$ is a left inductive $\FF_1$-linear monoid. However, $\beta, f, \alpha \in P(f)$ and $\beta f = \beta \neq 0$, $f \alpha = \alpha \neq 0$, whereas $\beta f \alpha = \beta \alpha = 0$. Thus, $P(f)$ is not a Brandt semigroup \cite{CliffPrest1961}. Scalar extension yields a similar example for $\hat{G}$-linear monoids.
\end{myeg} 

\begin{myeg}[Left Inductive $\not\Rightarrow$ Right Inductive]\label{e.LRCE}
Consider the quiver 
\[ Q = 
\begin{tikzcd} 
e \arrow[r, swap, bend right = 30, "\gamma"] & \arrow[l, swap, bend right = 60, "\alpha"] \arrow[l,bend right = 30, "\beta"] f
\end{tikzcd} 
\]  
subject to the (once again non-admissible) relations $\alpha \gamma = \beta \gamma = f$, $\gamma \alpha = \gamma \beta = 0$. The resulting $\FF_1$-linear path monoid with relations is $M = \{0, e, f, \alpha, \beta, \gamma, 1\}$. One can check that $P(e)$, $P(f)$, and $P(1)$ are $\FF_1$-linear representations under left translation, and hence $M$ is left inductive. However, $\alpha$, $\beta$, and $\gamma$ are non-zero elements of $P(f)$ satisfying $ \alpha \gamma = \beta \gamma = f \neq 0$, but $\alpha \neq \beta$. Thus, $P(f)$ is not $\FF_1$-linear under right translation, and $M$ is not right inductive. Similarly, one can construct an $\FF_1$-linear monoid that is right inductive, but not left inductive. Scalar extension then creates $\hat{G}$-linear examples.
\end{myeg}

%%%%%%%%%%%%%%%%%(SEMISIMPLE G-LINEAR MONOIDS)%%%%%%%%%%%%% 

\section{Semisimplicity Results for $\Rep(M,\hat{G})$}\label{s.SS}
 
We now characterize the finite, $\hat{G}$-linear, left inductive monoids $M$ for which $\Rep(M,\hat{G})$ is semisimple. The following result shows that $I_n(\hat{G})$ is left inductive and right inductive. Semigroup theorists should observe that this essentially follows from the fact that the principal factors of $I_n(\hat{G})$ are Brandt semigroups, and thus act on themselves via $\FF_1$-linear maps \cite{CliffPrest1961}. However, to keep the prerequisites of this article to a minimum, we have opted for a direct proof instead.

\begin{lem}\label{l.LIE}
For any finite abelian group $G$ and positive integer $n$, $I_n(\hat{G})$ is left inductive and right inductive. 
\end{lem}

\begin{proof} 
It is clear that $I_n(\hat{G})$ is a finite $\hat{G}$-linear monoid, so we only need to verify that it is left inductive and right inductive. We only prove left inductivity, as right inductivity follows from a similar argument. Every idempotent of $I_n(\hat{G})$ is of the form $E_S$, for a suitable subset $S\subseteq [n]$. Fix an $S$ and let $|S| = d$. Then it follows from Definition \ref{d.rank} that $J(E_S)$ is the set of all $A \in I_n(\hat{G})$ with $\operatorname{rank}(A) \le d$ and $A \in I(E_S)$ if and only if $\operatorname{rank}(A)<d$. In other words, $P(E_S) = I_n(\hat{G})^{\le d}/I_n(\hat{G})^{\le d-1}$ as a $\hat{G}$-vector space. It is also clear that as a $\hat{G}$-vector space,  
\[
I_n(\hat{G})^{\le d}/I_n(\hat{G})^{\le d-1} = \bigoplus_{|S| = d}{V_S} 
\] 
where the direct sum ranges over all $d$-element subsets $S$ of $[n]$ and $V_S := (I_n(\hat{G})^{\le d}/I_n(\hat{G})^{\le d-1})E_S$. So, it suffices to prove that $V_S \in \Rep(I_n(\hat{G}),\hat{G})$ for all $S$.
As a set, $V_S\setminus\{0\}$ is the collection of all matrices of the form $M_{S,f,c}$, where $f$ ranges over the injective functions $f : S \rightarrow [n]$ and $c$ ranges over the functions $c : f(S) \rightarrow G$. Suppose that $A \in I_n(\hat{G})$, $B, C \in V_S$ with $AB = AC \neq 0$. Recall that in Definition \ref{d.D_nG}, we defined $D_n(G)$ to be the set of all invertible diagonal matrices in $I_n(\hat{G})$. Since $I_n(\hat{G}) = D_n(G)\cdot I_n$, we can write $B = T_BM_{S,f_B}$ and $C = T_CM_{S,f_C}$, with $T_B, T_C \in D_n(G)$ and $M_{S,f_B}, M_{S,f_C} \in I_n$ satisfying $f_B(S) = f_C(S) = S'$. Write $A = T\cdot M_{S',f}$, where $T \in D_n(G)$ and $M_{S,f} \in I_n$. Then $A' = M_{f(S'),f^{-1}}T^{-1} \in I_n(\hat{G})$ satisfies $A'A = E_{S'}$. Thus, $B = E_{S'}B = (A'A)B = A'(AB) = A'(AC) = C$. It follows that $V_S$ is a $\hat{G}$-linear representation of $I_n(\hat{G})$. 
 \end{proof}

\begin{cor}\label{c.LRI}
Let $M$ be a finite, $\hat{G}$-linear inverse monoid. Then $M$ is left inductive and right inductive.
\end{cor} 

\begin{proof} 
By Theorem \ref{t.WP}, we may assume without loss of generality that $M$ is a $*$-closed submonoid of $I_n(\hat{G})$ for some positive integer $n$. Again, we only prove left inductivity, as right inductivity follows from a similar argument. Pick a non-zero idempotent $e \in M$. Then $e = E_S$ for some subset $S\subseteq [n]$ with $|S| = d$. Clearly, $J(e) = MeM \subset I_n(\hat{G})^{\le d}$. We claim that $I(e) = J(e) \cap I_n(\hat{G})^{\le d-1}$. It is clear that $I(e) \supseteq J(e)\cap I_n(\hat{G})^{\le d-1}$, so assume that $x \in J(e)$ satisfies $\operatorname{rank}(x) = d$. Write $x = ler$ for $l,r \in M$. Then $\operatorname{rank}(x) = \operatorname{rank}(e) = d$, and hence $\operatorname{rank}(le) = d$ as well. Since $(le)^* = e^*l^* = el^* \in eM$. Thus, $(le)^*(le)$ is a rank-$d$ idempotent in $eMe \subseteq eI_n(\hat{G})e = E_SI_n(\hat{G})E_S$. But the only rank-$d$ idempotent in $E_SI_n(\hat{G})E_S$ is $E_S = e$, and hence $(le)^*(le) = e$. It follows that $(le)^*x = er$, and similarly we may show $(le)^*x(er)^* = e$. In other words, $x \in J(e)\setminus I(e)$. It follows that $I(e) \subseteq J(e)\cap I_n(\hat{G})^{\le d}$. In other words, $P(e)$ is an $M$-subrepresentation of $[I_n(\hat{G})^{\le d}/I_n(\hat{G})^{\le d-1}]{\downarrow_e}$, which is $\hat{G}$-linear by Lemma \ref{l.LIE}. Thus, $M$ is left inductive. 
\end{proof}   

\begin{lem}\label{l.nilideal}
Let $M$ be a finite, $\hat{G}$-linear monoid and suppose that $\Rep(M,\hat{G})$ is semisimple. If $V \in \Rep(M,\hat{G})$ and $J$ is a nilpotent ideal of $M$, then $JV = 0$.
\end{lem} 

\begin{proof}  
Without loss of generality assume that $J \neq 0$, and let $d$ be the smallest positive integer for which $J^d = 0$. Note that $d \geq 2$. Since $JV \le V$, by Proposition \ref{p.verifyss} it follows that $V = JV \oplus W$ for a suitable $W \le V$. But then $J^{d-1}V = J^dV \oplus J^{d-1}W = J^{d-1}W$, and hence $J^{d-1}V = J^{d-1}W \subseteq JV \cap W = 0$. In general, if we have chosen $k$ so that $d-k \geq 2$ and we know that $J^{d-k}V = 0$, then $J^{d-k-1}V = J^{d-k}V\oplus  J^{d-k-1}W = J^{d-k-1}W \subseteq JV\cap W = 0$. Repeating this process a finite number of times allows us to conclude $JV = 0$, as we wished to show.\end{proof}

\begin{lem}\label{l.premain} 
Let $M$ be a finite, $\hat{G}$-linear, left inductive monoid. Then the following are equivalent: 
\begin{enumerate}
\item $M$ is inverse. 
\item $M$ is regular. 
\end{enumerate}
\end{lem} 

\begin{proof} 
The implication (1)$\Rightarrow$(2) follows from Proposition \ref{p.BasicInverses}. To prove (2)$\Rightarrow$(1), choose any $a \in M$. Then $J(a) = J(e)$ for some idempotent $e \in M$ and $P(a) = P(e)$ is $\hat{G}$-linear. If $a$ is non-zero, then $P(a)$ is $0$-simple and $P(a)^2 = P(e)^2 = P(e) \neq 0$ by simplicity. In particular, we may write $e = xayaz$ for suitable $x,y,z \in M$. It follows that $ayaz \in J(a)\setminus I(a) = J(e)\setminus I(e)$, and so $a$ does not act via the zero map on $P(e)$. Therefore,
\[ 
V = \bigoplus_{e \in E(M)}{P(e)}
\] 
is a faithful $\hat{G}$-linear representation of $M$. This is equivalent to the existence of an injective, $\hat{G}$-linear monoid homomorphism $M \hookrightarrow \End_{\hat{G}}(V) \cong I_n(\hat{G})$, where $n = \dim_{\hat{G}}(V)$. But $I_n(\hat{G})$ is an inverse monoid, and so its idempotents commute. Therefore, $M$ is a regular monoid whose idempotents commute, and so it is inverse by Proposition \ref{p.BasicInverses}. 
\end{proof}

\begin{mythm}\label{t.main} 
Let $M$ be a finite, $\hat{G}$-linear, left inductive monoid. Then the following are equivalent: 
\begin{enumerate} 
\item $\Rep(M,\hat{G})$ is semisimple. 
\item For all nilpotent ideals $J$ of $M$ and $V \in \Rep(M,\hat{G})$ we have $JV = 0$. 
\end{enumerate}  
In particular, if $M$ is also regular then $\Rep(M,\hat{G})$ is semisimple.
\end{mythm} 

\begin{proof} 
 (1)$\Rightarrow$(2) follows from Lemma \ref{l.nilideal}. To prove (2)$\Rightarrow$(1), assume that for any nilpotent ideal $J$ and $V \in \Rep(M,\hat{G})$ we have $JV = 0$. Let $\mathcal{R}$ denote the union of all nilpotent ideals of $M$. Then $\mathcal{R}$ is a proper ideal of $M$, and since $M$ is a finite set we can write $\mathcal{R} = J_1\cup \ldots \cup J_k$, where each $J_i$ is a nilpotent ideal of $M$. Since $J_{i_1}J_{i_2} \subseteq J_{i_1} \cap J_{i_2}$ for all $i_1, i_2 \le k$, it follows that $\mathcal{R}$ is itself a nilpotent ideal of $M$.

Consider the Rees factor monoid $M/\mathcal{R}$. Since $M$ is finite and $\hat{G}$-linear, Remark \ref{r.ReesFactor} implies that $M/\mathcal{R}$ is as well. We claim that $M/\mathcal{R}$ is regular and left inductive. To see that it is regular, suppose by way of contradiction that there exists an element $a \in M$ for which $[a] \in M/\mathcal{R}$ is not regular. Then in particular $a$ is not regular in $M$ and $[a]\neq [0]$ in $M/\mathcal{R}$, which implies $a \not\in \mathcal{R}$ in $M$. But Proposition \ref{p.nilpotentcondition} then implies that $J(a)$ is nilpotent and hence $J(a) \subseteq \mathcal{R}$, a contradiction. Thus, $M/\mathcal{R}$ is regular. To see it is left-inductive, choose an idempotent $e \not\in \mathcal{R}$ and consider the principal factor $P([e]) = J([e])/I([e])$ in $M/\mathcal{R}$. Suppose that $x, m_1, m_2 \in M$ are chosen so that $[m_1], [m_2] \in P([e])$ and $[x][m_1] = [x][m_2] \neq [0]$. Since $[x][m_i] = [xm_i]$ for $i = 1,2$ this implies $xm_1 = xm_2$ and $xm_1, xm_2 \not\in \mathcal{R}$. But $[m_i] \in P([e])$ implies that $m_1$ and $m_2$ generate $J(e)$ in $M$, and hence $xm_1 = xm_2 \not\in I(e)$. Since $P(e)$ is a $\hat{G}$-linear representation of $M$ under left translation, it follows that $m_1 = m_2$, and hence $[m_1] = [m_2]$. Thus, $M/\mathcal{R}$ is left inductive.  

It now follows from Lemma \ref{l.premain} that $M/\mathcal{R}$ is inverse. We claim that $\Rep(M/\mathcal{R},\hat{G})$ is semisimple. Indeed, pick $V \in \Rep(M/\mathcal{R},\hat{G})$ and $W \le V$. Suppose that $v \in V\setminus W$ and $[a]v \in W$ for some $[a] \in M/\mathcal{R}$. Then $([a]^*[a])v = [a]^*([a]v) \in W\cap \{0,v\} = \{0\}$, and hence $[a]v = ([a][a]^*[a])v = 0$ as well. In other words, for all $v \in V\setminus W$ and $[a] \in M/\mathcal{R}$, either $[a]v \in V\setminus W$ or $[a]v = 0$. It follows that $V/W = (V\setminus W)\cup \{0\}$ is a complement to $W$ in $V$. Thus, $V$ is completely reducible, which from Proposition \ref{p.verifyss} shows that $\Rep(M/\mathcal{R},\hat{G})$ is semisimple. 

Finally, given $V \in \Rep(M,\hat{G})$ and $W \le V$, (2) implies $\mathcal{R}W =  \mathcal{R}V = 0$ since $\mathcal{R}$ is nilpotent. Thus,  $V$ may be considered as a $\hat{G}$-linear representation of $M/\mathcal{R}$, with $W$ a sub-representation. By semisimplicity, $V = W\oplus W'$ for a suitable subspace $W' \subseteq V$ with $[a]W' \subseteq W'$ for all $[a] \in M/\mathcal{R}$. But note that $\mathcal{R}W' \subseteq \mathcal{R}V = 0$, and for all $a \in M\setminus\mathcal{R}$ we have $aW' = [a]W' \subseteq W'$. It follows that $V = W\oplus W'$ as $M$-representations, and thus $\Rep(M,\hat{G})$ is semisimple.
\end{proof}

We note the following useful consequence of the previous result:

\begin{cor} 
Let $M$ be a finite, $\hat{G}$-linear, regular monoid. Then the following are equivalent: 
\begin{enumerate} 
\item $M$ is left inductive.
\item $M$ is right inductive. 
\item $M$ is inverse.
\end{enumerate}
\end{cor} 

\begin{proof} 
The equivalence (1)$\Leftrightarrow$(3) follows immediately from Lemma \ref{l.premain}, so we will be done if we can prove (1)$\Leftrightarrow$(2). Again by Lemma \ref{l.premain}, (1) implies that $M$ is an inverse monoid, and so (2) follows from Corollary \ref{c.LRI}. Conversely, if (2) holds then $M^{\operatorname{op}}$ is inverse by Lemma \ref{l.premain}, which implies that $M$ itself is inverse. Thus, (1) holds by Corollary \ref{c.LRI} again.
\end{proof}

%\begin{mythm} 
%Let $M$ be a finite, $\hat{G}$-linear inverse monoid. Then $\Rep(M,\hat{G})$ is semisimple.
%\end{mythm} 

%\begin{proof} 
%Let $V \in \Rep(M,\hat{G})$ with $W\le V$. Suppose $v \in V\setminus W$ and $av \in W$ for some $a \in M$. Then $(a^*a)v = a^*(av) \in W\cap \{0,v\} = \{0\}$, and hence $av = (aa^*a)v = 0$ as well. In other words, for all $v \in V\setminus W$ and $a \in M$, either $av \in V\setminus W$ or $av = 0$. It follows that $V/W = (V\setminus W)\cup \{0\}$ is a complement to $W$ in $V$. Thus, $V$ is completely reducible, which shows that $\Rep(M,\hat{G})$ is semisimple.
%\end{proof}  

\begin{myeg}[A non-regular inductive monoid] 
Fix $G$ and $n$. Consider the $\hat{G}$-linear monoid $M$ with $\hat{G}$-basis $\{x_1,\ldots , x_n, 1\}$, such that $x_ix_j = 0$ for all $1 \le i ,j \le n$. In other words, the non-zero elements of $M$ are of the form $gx_i = x_ig$ or $g$, for $g \in G$ and $1\le i \le n$. Then $M$ admits an ideal filtration 
\[ 
M_0 = \{0\} \subset M_1 = \bigoplus_{i=1}^n{\hat{G}x_i} \subset M_2 = M
\] 
Note that the only non-zero regular principal factor is $P(1) = M/M_1 \cong M^{\times}$, which is $\hat{G}$-linear as an $M$-representation. Thus, $M$ is left (and right) inductive. As expected, $\Rep(M,\hat{G})$ is not semisimple: $M$ is a $\hat{G}$-linear representation over itself, and $M_1$ is a non-zero, proper subrepresentation that does not admit a complement.
\end{myeg}

\begin{myeg}[$\hat{G}$-linear representations of $I_n(\hat{G})$]\label{e.ISMReps} $I_n(\hat{G})$ is a left-inductive $\hat{G}$-linear inverse monoid, and so $\Rep(I_n(\hat{G}),\hat{G})$ is semisimple by Theorem \ref{t.main}. Theorem \ref{t.CMP} asserts that every simple representation of $I_n(\hat{G})$ is isomorphic to $W{\uparrow_e}/N(W{\uparrow_e})$, where $e \in I_n(\hat{G})$ is an idempotent with $\mathcal{J}$-class $J$ and $W$ is a simple $\hat{G}$-linear representation of $G_J$. Now, $e = E_S$ for some subset $S \subset [n]$ with $|S| = d$, and hence $J = I_n(\hat{G})^{\le d}\setminus I_n(\hat{G})^{\le d-1}$ is the set of all rank-$d$ matrices in $I_n(\hat{G})$. Thus, $G_J = eI_n(\hat{G})e\cap J = \{ M_{S, \sigma, c} \mid \sigma \in \Symm(S)\} \cong G^d\rtimes S_d$ under the product $(c,\sigma )(d, \tau) = (c\times (d\circ \sigma^{-1}), \sigma \tau)$. A $\hat{G}$-linear representation of a group is in particular $\FF_1$-linear, and hence a simple representation $W \in \Rep(\hat{G_J},\hat{G})$ is just an orbit space $(G^d\rtimes S_d)/H$ which is a free $G$-set under left translation. The condition that $(G^d\rtimes S_d)/H$ be a free $G$-set is equivalent to $G \cap H$ being trivial, where $G$ is identified with the diagonal subgroup of $G^d$. Fixing such an $H$, we have a simple $\hat{G}$-linear $\hat{G}_J$-representation $W = (G^d\rtimes S_d)/H\cup \{0\}$. Now, $W{\uparrow_e}$ can be identified with the quotient of the representation $P(e)e$ by the equivalence relation $M_{S,f,c} \sim_H M_{S,g,d}$ if and only if $(g,d)^{-1}(f,c) \in H$. Under this identification, it is straightforward to verify that $N(W{\uparrow_e}) = \{0\}$ and hence $W{\uparrow_e}$ is already simple. We abbreviate the representation $W{\uparrow_e}$ as $V_{S,H}$. One can prove that $V_{S,H} \cong V_{T,K}$ if and only if there exists a bijection $\psi : T \rightarrow S$ and a $t \in G^S$ such that the induced isomorphism 
\[ 
F_{\psi,t}: G^S\rtimes \Symm(S) \rightarrow G^T\rtimes \Symm(T) 
\] 
\[ 
\displaystyle (c,f)\mapsto \left(\left(\frac{t\circ f^{-1}\psi}{t\circ \psi}\right)(c\circ \psi), \psi^{-1}f\psi\right)
\] 
satisfies $F_{\psi,t}(H) = K$. The reader is encouraged to compare these results to the construction of row-monomial representations in \cite{Tully1964}.
\end{myeg} 

\begin{myeg}[A non-regular $M$ with $\Rep(M,\hat{G})$ semisimple] 
The condition $\mathcal{R}V = 0$ for all $V \in \Rep(M,\hat{G})$ is vacuous if $M$ is regular, as $\mathcal{R} = 0$ in that case. We now exhibit a finite, $\hat{G}$-linear, left inductive monoid $M$ for which $\Rep(M,\hat{G})$ is semisimple and $\mathcal{R} \neq 0$. Let $M$ be the following collection of block-diagonal matrices: 
\[ 
M = \bigg\{ \left[ \begin{array}{cc} A & 0 \\ 0 & 0 \end{array}  \right] \bigg\}_{A \in I_2(\hat{G})}\cup \bigg\{\left[ \begin{array}{cc} 0 & 0 \\ 0 & g\cdot E_{12} \end{array} \right], \left[ \begin{array}{cc} g\cdot I & 0 \\ 0 & g\cdot I \end{array} \right] \bigg\}_{g \in G}
\] 
where $I$ is the $2\times 2$ identity matrix and $E_{12} = \left[ \begin{array}{cc} 0 & 1 \\ 0 & 0 \end{array} \right] \in I_2(\hat{G})$. Then $M$ is a $\hat{G}$-linear, left inductive submonoid of $I_4(\hat{G})$ and the subset $S:= M\setminus\bigg\{ \left[ \begin{array}{cc} 0 & 0 \\ 0 & g\cdot E_{12} \end{array} \right] \bigg\}_{g \in G}$ is a $*$-closed submonoid, hence an inverse monoid. Note that $M$ is not regular, as the element $W = \left[ \begin{array}{cc} 0 & 0 \\ 0 & E_{12} \end{array} \right]$ satisfies $XW = 0 = WX$ for all non-invertible $X$. In particular, we have $\mathcal{R} = \{0\} \cup \{ g\cdot W\mid g \in G\}$. We claim that $\mathcal{R}V = 0$ for all non-zero $V \in \Rep(M,\hat{G})$. Indeed, $V$ can be decomposed into a direct sum $V = V_1\oplus \cdots \oplus V_k$, where each $V_i$ is a simple $\hat{G}$-linear representation of $S$. Then each $V_i$ has an apex $\mathcal{J}_i$. If $\mathcal{J}_i$ contains a unit then $V_i$ is annihilated by every non-unit of $M$, and in particular $WV_i = 0$. Otherwise $\mathcal{J}_i$ contains a non-zero, non-unit idempotent $e_i$ and there exists a $v_i \in V_i$ with $e_iv_i = v_i$. For any non-zero element $w_i \in V_i$ with $v_i \neq w_i$, there exists an $x \in S$ with $w_i = xv_i$. Then $w_i = (xx^*)w_i = (xx^*x)v_i = xv_i = w_i$, and hence $w_i$ is fixed by the non-unit idempotent $xx^*$. Since $We = 0 = eW$ for all $e \in E(M)$ with $e \neq 1_M$, it follows that $WV_i = 0$ in this case as well. In other words, $WV = 0$ and so $\mathcal{R}V = 0$. From Theorem \ref{t.main}, it follows that $\Rep(M,\hat{G})$ is semisimple.
\end{myeg}

%%%%%%%%%%%%%%%%%%%%%%%%%%%%%%%%%%%%%%%%%%%%%%%%%% 
%%%%%%%%%%%%%%%%( OTHER PERSPECTIVES )%%%%%%%%%%%%%%%%%%%% 
%%%%%%%%%%%%%%%%%%%%%%%%%%%%%%%%%%%%%%%%%%%%%%%%%% 

\section{Other Perspectives on Representations of Submonomial matrices}\label{s.Alt}

Naively, one might wish to treat $I_n$ as ``$n\times n$ matrices defined over $\FF_1$,'' with $I_n(\hat{G})$ playing an analogous role for different choices of scalars. In many ways, this is an appropriate identification: for any finite abelian $\hat{G}$, the $I_n(\hat{G})$'s are endomorphism monoids of $\hat{G}$-vector spaces, and so if $M$ is a finite $\hat{G}$-linear monoid, then functors $M \rightarrow \Vect_{\hat{G}}$ are the same thing as monoid homomorphisms $M \rightarrow I_n(\hat{G})$ that commute with scalar multiplication (for an appropriate $n$). Furthermore, when we specialize to $G = \{1\}$, the group of units of $I_n$ is the Weyl group of $\operatorname{GL}_n(k) = \M_n(k)^{\times}$, as would be expected from the theory of algebraic groups over $\FF_1$ \cite{LorAlgGroups2011}. However, there are instances where this analogy falters: notably, the $\FF_1$-representation theory of $I_n$ admits significant differences from the $k$-linear representations of $\M_n(k)$. Specifically, $M_n(k)$ is Morita equivalent to $k$, and so its category of finite-dimensional $k$-representations is equivalent to the category $\Vect_k$ of finite-dimensional $k$-vector spaces. In particular, it has only one simple representation up to isomorphism. In contrast, Example \ref{e.ISMReps} exhibits many simple $\FF_1$-linear representations of $I_n$. This is reflected in the structure of the monomial algebra $k[I_n]$, where $M_n(k)$ is seen to only be an idempotent subalgebra of $k[I_n]$. In fact, one has the following result, which is a specialization of \cite[Corollary 9.4]{Steinberg2016}: 

\begin{mythm} 
For any field $k$ and any positive integer $n$,  
\[
\displaystyle k[I_n(\hat{G})] \cong \prod_{d=1}^n{M_{{{n}\choose{d}}}(k[G^d\rtimes S_d])} 
\]
 as $k$-algebras. In particular, $k[I_n(\hat{G})]$ is semisimple as a $k$-algebra if and only if either $\operatorname{char}(k) = 0$, or $\operatorname{char}(k)>n$ and $\operatorname{char}(k)\nmid |G|$. 
\end{mythm} 

Nevertheless, the $\FF_1$-representation theory of $I_n$ still retains some of the features we would expect from the theory of semisimple rings. Importantly, Theorem \ref{t.main} asserts that $\Rep(I_n,\FF_1)$ is a semisimple category, and the simple representation $\FF_1^n := V_{\{1\},\{1\}}$ is a direct $\FF_1$-analog of the simple representation $k^n$ of $\M_n(k)$. In fact, one can define a functor 
\[ 
F_k: \Rep(I_n,\FF_1) \rightarrow \Rep(\M_n(k),k)
\] 
satisfying and $F(\FF_1^n) = k^n$, $F(S) = 0$ for all other simples, and $F(V\oplus W) \cong F(V)\oplus F(W)$ for all $V, W \in \Rep(I_n,\FF_1)$. This functor is surjective on objects and faithful, but not full. Of course, we have analogous results for $I_n(\hat{G})$, where the simple representation $\FF_1^n$ is replaced with $\hat{G}^n$. 

In light of all of this, it is natural to ask whether the full proto-abelian subcategory of $\Rep(I_n,\FF_1)$ generated by $\FF_1^n$ can be recovered as the natural category of $\FF_1$-representations of some better-behaved substitute for $I_n$. Replacing $I_n$ with the sub-monoid $R_n:= I_n^{\le 1} \sqcup \{I_n\}$ works, because $\Rep(R_n, \FF_1) \cong \Vect_{\FF_1}$ as finitary proto-abelian categories and $k[R_n] \cong M_n(k)$. However, this is less useful than it might at first appear: although the correct category and monoid ring are recovered, $R_n^{\times}$ is trivial, and only rarely can functors $M\rightarrow \Vect_{\FF_1}$ be interpreted as monoid homomorphisms $M\rightarrow R_n$. 

What follows in this section are three attempts to build this better-behaved substitute by adding higher structure to $I_n(\hat{G})$: the first, and most successful, considers a poset structure on $I_n(\hat{G})$ that essentially recovers addition as a partial operation on the monoid. We end up with an object similar to a quantale, whose natural category of representations can be identified with $\Vect_{\hat{G}}$. Furthermore, these representations can be interpreted as representations over a certain quotient of a monoid algebra over a field $k$, which is isomorphic to $M_n(k)$ when $G = \{1\}$. This extra structure provides one explanation for the difference between classical and non-additive representation theory. 

 The following two attempts are not as successful as the first, but are included because they are likely to be of interest when studying how representations behave under scalar extension. The second employs the theory of bialgebras over $\ZZ$, and the third proceeds by interpreting $I_n$ as a functor on abelian groups. No attempt is made to provide an exhaustive approach to either perspective, and they will be explored more fully in a future work. For now, we restrict our goals to defining natural representation categories for each, demonstrating that each contains more than one simple object, and that each contains infinitely-many non-isomorphic objects.

%%%%%%%%%%%%%%%%(THE POSET STRUCTURE)%%%%%%%%%%%%%%%%%%%%%%%%%%%%%%%%%%%  

\subsection{Monoids with Partial Order}  

A curious impact of only treating $I_n(\hat{G})$ as a monoid is that it ignores relationships between its elements which are determined by the addition operation of $M_n(k[G])$. Of course, $I_n(\hat{G})$ is not closed under addition, so it only defines a partial operation on $I_n(\hat{G})$. Instead of taking this route, we are able to adequately capture the additive relationships through a poset structure on $I_n(\hat{G})$. This construction is similar to that of a unital quantale, but without requiring the poset in question to be a complete lattice (or even a join-semilattice).

\begin{mydef}
For $A, B \in I_n(\hat{G})$, we write $A \le B$ if and only if for all $x \in \hat{G}^{\oplus n}$, $Ax \neq 0$ implies $Bx = Ax$.
\end{mydef} 

The relation $\le $ defines a partial order on $I_n(\hat{G})$. Note that for any subset $I \subseteq I_n(\hat{G})$, $U := \bigvee_{A \in I}{A}$ exists if and only if for all $A, B \in I$, $x \in \operatorname{supp}(A)\cap \operatorname{supp}(B)$ implies $Ax = Bx$. Then $U$ is the element with $\operatorname{supp}(U) = \bigcup_{A \in I}{\supp (A)}$ defined on $x \in \operatorname{supp}(U)$ as $Ux = Ax$, for any $A \in I$ with $x \in \operatorname{supp}(A)$. Dually, $D:= \bigwedge_{A \in I}{A}$ always exists, as the element with 
\[
\operatorname{supp}(D) = \bigg\{x \in \bigcap_{A \in I}{\operatorname{supp}(A)} \mid Ax = Bx\text{ for all $A,B \in I$} \bigg\}
\]
 defined on $x \in \operatorname{supp}(D)$ as $Dx = Ax$ for any $A \in I$. Furthermore, for any $B \in I_n(\hat{G})$ we have 
\[ 
B\left(\bigvee_{A \in I}{A} \right) = \bigvee_{A \in I}{BA}
\] 
\[ 
B\left( \bigwedge_{A \in I}{A} \right) = \bigwedge_{A\in I}{BA}
\]
and similarly for right multiplication by $B$.  

\begin{mydef}\label{d.poset}
Suppose that $\mathcal{M} = (M,\preceq)$, where $M$ is a finite $\hat{G}$-linear monoid and $\preceq$ is a partial order on $M$ satisfying the following conditions: 
\begin{enumerate}  
\item $0_M$ is the unique minimal element of $M$ with respect to $\preceq$.
\item For all subsets $I \subseteq M$, if $\bigvee_{a \in I}{a}$ exists then  
\[ 
x\left(\bigvee_{a \in I}{a} \right) = \bigvee_{a \in I}{xa}
\]  
\[ 
\left(\bigvee_{a \in I}{a} \right)x = \bigvee_{a \in I}{ax}
\] 

\item For all subsets $I \subseteq M$, if $\bigwedge_{a \in I}{a}$ exists then  
\[ 
x\left(\bigwedge_{a \in I}{a} \right) = \bigwedge_{a \in I}{xa}
\]  
\[ 
\left(\bigwedge_{a \in I}{a} \right)x = \bigwedge_{a \in I}{ax}
\] 
\end{enumerate} 
Then the category $\Rep(\mathcal{M},\hat{G})$ is the full subcategory of $\Rep(M,\hat{G})$ whose objects are precisely those $V: M \rightarrow \Vect_{\hat{G}}$ satisfying the following: whenever $\bigvee_{a \in I}{a}$ exists, $\bigvee_{a \in I}{V(a)}$ exists and 
\[ 
\bigvee_{a \in I}{V(a)} = V\left(\bigvee_{a \in I}{a} \right).
\]
\end{mydef} 
For any such $\mathcal{M} = (M, \preceq)$, the zero representation is in $\Rep(\mathcal{M},\hat{G})$. Furthermore, we have the following: 

\begin{pro} 
Let $\mathcal{M} = (M,\preceq)$ be as in Definition \ref{d.poset}. Then $\Rep(\mathcal{M},\hat{G})$ is closed (as a subcategory of $\Rep(M,\hat{G})$) under sub-objects and quotients. In particular, $\Rep(\mathcal{M},\hat{G})$ is a proto-abelian category.
\end{pro}  

\begin{proof} 
Let $V \in \Rep(\mathcal{M},\hat{G})$, and let $W$ be a $\hat{G}$-linear $M$-subrepresentation of $V$. Fix an ordered basis $\{w_1,\ldots , w_d, v_{d+1},\ldots , v_n\}$ for $V$, where $\{w_1,\ldots , w_d\}$ is a basis for $W$ as a $\hat{G}$-vector space. Then $\{v_{d+1},\dots, v_n\}$ is an ordered basis for the quotient representation $V/W$, and with respect to this ordered basis, it is easy to check that 
\[ 
V(a) = \left[ \begin{array}{cc} W(a) & * \\ 0 & (V/W)(a) \end{array} \right]
\] 
for all $a \in M$. Suppose $u := \bigvee_{a \in I}{a}$ exists in $M$. By hypothesis, we have $V(u) = \bigvee_{a \in I}{V(a)}$ in $\End_{\hat{G}}(V) \cong I_n(\hat{G})$. Then for all $1 \le i \le d$, $W(u)(w_i) = V(u)(w_i) \neq 0$ implies that $W(u)(w_i) =V(u)(w_i) = V(a)(w_i) = W(a)(w_i)$ for all $a\in I$ with $V(a)(w_i) \neq 0$. But then $\bigvee_{a \in I}{W(a)}$ exists and $W(u) = \bigvee_{a \in I}{W(a)}$. Similarly, if $(V/W)(u)(v_j) \neq 0$, then $(V/W)(u)(v_j)  = V(u)(v_j) \neq 0$ and so $(V/W)(u)(v_j) = V(u)(v_j) = V(a)(v_j) = (V/W)(a)(v_j)$ for all $a \in I$ with $V(a)(v_j) \neq 0$. It follows that $\bigvee_{a \in I}{(V/W)(a)}$ exists and $(V/W)(u) = \bigvee_{a \in I}{(V/W)(a)}$. 
To verify that $\Rep(\mathcal{M},\hat{G})$ is proto-abelian, simply note that diagrams in Axioms (4) or (5) may be completed to bi-Cartesian squares in $\Rep(M,\hat{G})$ as in Axiom (3). Since $\Rep(\mathcal{M},\hat{G})$ is closed under sub-objects, $X \in \Rep(\mathcal{M},\hat{G})$ if $Y,X',Y' \in \Rep(\mathcal{M},\hat{G})$; since $\Rep(\mathcal{M},\hat{G})$ is closed under quotients, $Y' \in \Rep(\mathcal{M},\hat{G})$ if $X,Y,X' \in \Rep(\mathcal{M},\hat{G})$.
\end{proof} 

\begin{cor} 
Let $\mathcal{M} = (I_n(\hat{G}),\le)$. Then $\Rep(\mathcal{M},\hat{G})$ is the full subcategory of $\Rep(I_n(\hat{G}),\hat{G})$ whose objects are direct sums of the simple object $\hat{G}^{\oplus n}$.
\end{cor} 
\begin{proof} 
Let $V \in \Rep(\mathcal{M},\hat{G})$. Since $I_n = \bigvee_{i=1}^n{E_{ii}}$ in $I_n(\hat{G})$, we must have $\id_V = V(I_n) = \bigvee_{i=1}^n{V(E_ii)}$, and so some $E_{ii}$ acts via a non-zero map on $V$. But this must remain true for any sub-representation of $V$, in particularly for any simple direct summand of $V$. But then every simple direct summand of $V$ must be isomorphic to $\hat{G}^{\oplus n}$. Since any direct sum of $\hat{G}^{\oplus n}$ is clearly in $\Rep(\mathcal{M},\hat{G})$, the claim follows.
\end{proof}   

When $G = \{1\}$, the corollary above tells us that the category generated by $\FF_1^{\oplus n}$ is the category of $\FF_1$-representations of the object $(I_n,\le )$.

By condition (1) of Definition \ref{d.poset}, $1 = 1\vee 0$ in $M$. A \emph{complete system of orthogonal idempotents} in $M$ is a finite set $\{e_1,\ldots, e_n\}$ such that $e_ie_j = \delta_{ij}e_j$ for $1\le i,j \le n$ and $1 = \bigvee_{i=1}^n{e_i}$. Note that $e_i^2 = e_i$ is redundant, as condition (2) along with orthogonality implies $e_j = e_j\cdot 1 = e_j\left(\bigvee_{i=1}^n{e_i} \right) = \bigvee_{i=1}^n{e_je_i} = 0\vee e_j^2 = e_j^2$ for all $1\le j \le n$. 

Let $\mathcal{C}(M)$ denote the set of all complete systems of orthogonal idempotents in $M$. Since $M$ is finite and $\{0,1\} \in \mathcal{C}(M)$, $\mathcal{C}(M)$ is a finite non-empty set. 

\begin{mydef} 
Let $\mathcal{M} = (M,\preceq)$ be as in Definition \ref{d.poset}. For a field $k$, define 
\[ 
k[\mathcal{M}] := k[M]/\left( a = \sum_{e \in E}{ae} = \sum_{e \in E}{ea} \mid a\in M, E \in \mathcal{C}(M) \right)
\]  
where $k[M]$ denotes the usual monoid algebra of $M$.
\end{mydef} 

Note that for any $\hat{G}$-linear monoid $M$, $V \in \Rep(M,\hat{G})$, and field $k$, the free $k$-vector space on $V\setminus\{0\}$ is a representation of the monoid algebra $k[M]$ in a natural way. Thus, we have an extension of scalars functor 
\[ 
-\otimes k : \Rep(M,\hat{G}) \rightarrow k[M]-\operatorname{mod}
\]

\begin{pro} 
Let $\mathcal{M} = (M,\preceq)$ be as in Definition \ref{d.poset}. For any field $k$, the extension of scalars functor 
\[ 
-\otimes k : \Rep(M,\hat{G}) \rightarrow k[M]-\operatorname{mod}
\]  
restricts to a functor 
\[ 
-\otimes k : \Rep(\mathcal{M},\hat{G}) \rightarrow k[\mathcal{M}]-\operatorname{mod}
\]  
\end{pro}  

\begin{proof} 
Note that $E \in \mathcal{C}(I_n(\hat{G}))$ if and only if $I_n = \sum_{e \in E}{e}$ in $M_n(k[G])$. If $V \in \Rep(\mathcal{M},\hat{G})$ has structure map $\rho : M \rightarrow \End_{\hat{G}}(V)$, then for all $a \in A$ and $E \in \mathcal{C}(M)$ we have $I_n = \rho(1) = \rho \left(\bigvee_{e \in E}{e}\right) = \bigvee_{e \in E}{\rho(e)} = \sum_{e \in E}{\rho(e)}$ and hence $\rho(a) = \rho(a)I_n = \rho(a)\sum_{e \in E}{\rho(e)} = \sum_{e \in E}{\rho(a)\rho(e)} = \sum_{e \in E}{\rho(ae)}$ for all $a \in M$. Similarly, $\rho(a) = \sum_{e \in E}{\rho(ea)}$ and so the $k$-algebra map $k[M] \rightarrow \End_k(V\otimes k)$ descends to a map $k[\mathcal{M}] \rightarrow \End_k(V\otimes k)$.
\end{proof} 

When $\mathcal{M} = (I_n,\le )$ it is easy to check that $k[\mathcal{M}] \cong M_n(k)$ and so we have a functor  
\[
\Rep(\mathcal{M},\FF_1) \rightarrow M_n(k)-\operatorname{mod} 
\]
 that is faithful and surjective on objects. In this way, we can recover the representation theory of $M_n(k)$ from the $\FF_1$-representation theory of $(I_n,\le )$.

%%%%%%%%%%(G-LINEAR REPS AS BIALGEBRAS)%%%%%%%%%%%%%%%%%%%%% 

\subsection{Bialgebras of Submonomial Matrices over Pointed Abelian Groups}   

In this section we consider symmetric inverse monoids as certain affine monoid schemes. Our constructions are well-known from the theory of Hopf algebras and affine group schemes, see for instance \cite{Waterhouse}. Let $\operatorname{CRing}$ denote the category of commutative rings with $1$, and $\operatorname{Mon}_0$ the category of monoids with absorbing element. For a fixed commutative ring $R$, let $R$-$\operatorname{CAlg}$ denote the category of all commutative $R$-algebras. In particular, $\operatorname{CRing} = \mathbb{Z}$-$\operatorname{CAlg}$. If $\M_n(\ZZ)$ denotes the $n\times n$-matrices, considered as an affine scheme over $\ZZ$, then $\M_n(\ZZ) \cong \mathbb{A}^{n^2}$ and so the coordinate ring of $\M_n(\ZZ)$ is simply the polynomial ring in $n^2$ variables
\[
\ZZ[\M_n(\ZZ)] = \ZZ[x_{ij} \mid 1 \le i,j \le n]. 
\]  
The coordinate ring $\ZZ[\M_n(\ZZ)]$ is a bialgebra over $\ZZ$ with comultiplication $\Delta (x_{ij}) = \sum_k{x_{ik}\otimes x_{kj}}$ and counit $\epsilon (x_{ij}) = \delta_{ij}$. The ideal $\mcI(\MM_n(\ZZ)) = \langle x_{ik}x_{jk}, x_{ki}x_{kj} \mid 1 \le k \le n, 1 \le i < j \le n \rangle$ is in fact a biideal, meaning that $\Delta(\mcI(\MM_n(\ZZ))) \subseteq \mcI(\MM_n(\ZZ))\otimes \ZZ[\M_n(\ZZ)] + \ZZ[\M_n(\ZZ)]\otimes\mcI(\MM_n(\ZZ))$. Hence the quotient algebra 
\[ 
\ZZ[\MM_n(\ZZ)] = \ZZ[\M_n(\ZZ)]/\mcI(\MM_n(\ZZ))
\]
inherits the structure of a bialgebra over $\ZZ$. We let $\MM_n(\ZZ) := \operatorname{CRing}(\ZZ[\MM_n(\ZZ)], -)$ denote the corresponding functor $\operatorname{CRing} \rightarrow \operatorname{Mon}_{0}$. If $G$ is a finite multiplicative abelian group, then $\MM_n(\hat{G})$ will denote the functor $\ZZ[G]$-$\operatorname{CAlg}\rightarrow \operatorname{Mon}_0$ obtained by extension of scalars along the natural inclusion $\ZZ \hookrightarrow \ZZ[G]$. Note that the bialgebra representing $\MM_n(\hat{G})$ is simply $\ZZ[G]\otimes_{\ZZ}\ZZ[\MM_n(\ZZ)]$. 

\begin{convention} 
For the rest of this section, fix a finite multiplicative abelian group $G$. To ease notation, we define $R:= \ZZ[G]$ and abbreviate $\MM_n(\hat{G})$ as $\MM_n$. Thus, the $R$-bialgebra representing $\MM_n$ will be denoted as $R[\MM_n]$.
\end{convention}

For any $A \in R$-$\operatorname{CAlg}$, a point $f \in \MM_n(A)$ is a $R$-algebra homomorphism $f : R[\MM_n] \rightarrow A$. This is specified by an $n\times n$ matrix $\left( f_{ij} \right)$ with coefficients in $A$ satisfying $f_{ik}f_{jk} = 0 = f_{ki}f_{kj}$ for all $k$ and $i \neq j$. Clearly, this includes all submonomial matrices with entries in $A$. If $G = \{1\}$ and $A$ is an integral domain, any $A$-point of $\MM_n$ arises from a submonomial matrix.   

\begin{rmk} 
If $T$ is a finite set with $|T| = d$, we let $\MM_T \cong \MM_d$ be the affine scheme of submonomial matrices whose rows and columns are indexed by $T$.
\end{rmk}

\begin{mydef} 
A \emph{$\hat{G}$-linear representation of $\MM_n$} of dimension $d$ is an $R$-bialgebra map  
\[ 
\rho: R[\MM_d]\rightarrow R[\MM_n] 
\]
 satisfying the following: for any $R$-algebra homomorphism $m : R[\MM_n] \rightarrow R$ satisfying $m(x_{ij}) \in \hat{G}$ for all $1 \le i,j \le n$, the $R$-algebra homomorphism $m\circ \rho : R[\MM_d]\rightarrow R$ satisfies $[m\circ \rho](x_{kl}) \in \hat{G}$ for all $1\le k, l \le d$. 
\end{mydef}  

If $\rho : R[\MM_d] \rightarrow R[\MM_n]$ is a $\hat{G}$-linear representation of $\MM_n$, there is an associated natural transformation $\tilde{\rho}: \MM_n \rightarrow \MM_d$ defined by $\tilde{\rho}_A(f) = f\circ \rho$ for each $f : R[\MM_n] \rightarrow A$. Let $\iota : R \rightarrow A$ denote the structure map of $A$. Suppose that $\iota$ is injective and that $f(x_{ij}) \in \iota(\hat{G})$ for all $1\le i,j \le n$. Then $f$ can be written as $f = \iota \circ \tilde{f}$, where $\tilde{f} : R[\MM_n] \rightarrow R$ is an $R$-algebra homomorphism satisfying $\tilde{f}(x_{ij}) \in \hat{G}$ for all $1 \le i,j \le n$. But then $[\tilde{f} \circ \rho](x_{kl}) \in \hat{G}$ for all $1 \le k,l \le d$, and hence $[f\circ \rho](x_{kl}) \in \iota(\hat{G})$ for all $1 \le k,l \le d$. In other words, $\tilde{\rho}_A$ carries the subset $I_n(\iota(\hat{G})) \subset \MM_n(A)$ into the subset $I_d(\iota(\hat{G})) \subset \MM_d(A)$. In particular, $\tilde{\rho}$ defines a $d$-dimensional $\hat{G}$-linear representation of $I_n(\iota(\hat{G}))$ on the set $\{\mathbf{0}\} \cup (\iota(G)\cdot\mathbf{e}_1) \cup\cdots \cup (\iota(G)\cdot\mathbf{e}_d) \subseteq A^d$. Here, $\mathbf{e}_j$ denotes the usual column vector with $1_A$ in the $j^{th}$ entry and $0_A$ in the others. 

For any positive integer $d$, we have a functor $\Aff^d: R$-$\operatorname{CAlg}\rightarrow \operatorname{Group}$ that carries $A$ to the additive group on $A^d$. Its representing Hopf algebra is $R[\Aff^d]:= R[x_1,\ldots , x_d]$, with each $x_i$ primitive. Of course, this is just the base extension of usual affine $d$-space to $R$. Matrix multiplication yields a natural transformation 
\[ 
\mu_d: \MM_d\times \Aff^d \rightarrow \Aff^d
\]  
\begin{mydef} 
Given any vector $\mathbf{v} = \left[\begin{array}{c} v_1 \\ v_2 \\ \vdots \\ v_n \end{array}\right] \in R^d$, we have an associated $R$-algebra morphism, also denoted $\mathbf{v}$:
\[ 
\mathbf{v}: R[\Aff^d] \rightarrow R
\] 
defined by $\mathbf{v}(x_i) = v_i$ for all $1\le i \le d$. With this identification, the set  
\[
V_d:= \{\mathbf{0}\} \cup \{ g\mathbf{e}_i \mid g \in G, 1 \le i \le d\} \subset \Aff^d(R) 
\]
 is a $\hat{G}$-linear vector space of dimension $d$ under the obvious $G$-action. %The \emph{$i^{th}$ standard basis vector} of $\Aff^d_G$ is the $R_G$-algebra morphism $e^{d}_i : R_G[\Aff^d_G] \rightarrow R_G$ associated to $\mathbf{e}_i \in R_G^n$. The \emph{zero vector} of $\Aff^d$ is the morphism $z^d : R_G[\Aff^d_G]\rightarrow R_G$ associated to . The set $\{ z^d,e_1^d,\ldots , e_d^d\}$ is an $\FF_1$-vector space of dimension $d$ with zero vector $z^d$.
\end{mydef}
\begin{mydef} 
Let $\rho : R[\MM_a] \rightarrow R[\MM_n]$ and $\phi : R[\MM_b]\rightarrow R[\MM_n]$ be two $\hat{G}$-linear representations of $\MM_n$. Let $f : \Aff^a \rightarrow \Aff^b$ be a natural transformation with dual map $f^* : R[\Aff^b] \rightarrow R[\Aff^a]$. Then $f$ is called a \emph{morphism from $\rho$ to $\phi$} if it satisfies the following: 
\begin{enumerate} 
\item The following diagram commutes:
\[
\begin{tikzcd}  
\MM_n\times \Aff^a \arrow{r}{\tilde{\rho}\times \id} \arrow[swap]{d}{\id\times f} & \MM_a \times \Aff^a \arrow{r}{\mu_a} & \Aff^a \arrow{d}{f} \\ 
\MM_n\times \Aff^b \arrow[swap]{r}{\tilde{\phi}\times \id}  & \MM_b\times \Aff^b \arrow[swap]{r}{\mu_b} & \Aff^b\\
\end{tikzcd} 
\]  
\item The map $f_{R}: \Aff^a(R) \rightarrow \Aff^b(R)$, $v \mapsto v\circ f^*$ induces a $\hat{G}$-linear map $V_a \rightarrow V_b$.
\end{enumerate}
\end{mydef}   

With this, the set of $\hat{G}$-linear representations of $\MM_n$ becomes a category $\Rep(\MM_n,\hat{G})$. Note that $\Rep(\MM_n,\hat{G})$ always has at least one $\hat{G}$-linear representation of dimension $n$, given by the identity map $R[\MM_n] \rightarrow R[\MM_n]$. This corresponds to the usual action of $I_n(\hat{G})$ on $\hat{G}\otimes_{\FF_1}[n]$. To find more examples, let $N$ be a subgroup of $G^n\rtimes S_n$ of index $d$. For $a \in G^n\rtimes S_n$, write $\bar{a} := aN$. We then consider the affine scheme $\MM_{(G^n\rtimes S_n)/N} \cong \MM_d$ of submonomial matrices indexed by the elements of $(G^d\rtimes S_n)/N$. 

\begin{mythm}\label{t.newexamples}
Suppose that $H$ is a normal subgroup of $S_n$. Then there is a $\hat{G}$-linear representation  
\[ 
\varphi_H : R[\MM_{(G^n\rtimes S_n)/(G^n\rtimes H)}] \rightarrow R[\MM_n] 
\]
of $\MM_n$ of dimension $[S_n:H]$ satisfying 
\[ 
\displaystyle \varphi_H(x_{\bar{a} \bar{b}}) = \sum_{\sigma \in P\left(\overline{ab^{-1}}\right)}{x_{\sigma}}
\] 
for all $\bar{a}, \bar{b} \in (G^n\rtimes S_n)/(G^d\rtimes H)$, where $P\left(\overline{ab^{-1}}\right) := \bigg\{ \sigma \in S_n \mid G^n\times \{\sigma\} \subseteq \overline{ab^{-1}} \bigg\}$ and $\displaystyle x_{\sigma}:= \prod_{k=1}^n{x_{\sigma(k) k}}$ for any $\sigma \in S_n$.
\end{mythm} 

\begin{proof} 
Since $G^n\rtimes H$ is normal in $G^n\rtimes S_n$ and $(c,\sigma ) \in G^n\rtimes H$ implies $G^n\times\{\sigma\} \subseteq G^n\rtimes H$, the formula for $\varphi_H$ is well-defined on the elements $x_{\bar{a}\bar{b}}$. Note that if $\bar{b} \neq \bar{c}$, then $\overline{ab^{-1}} \cap \overline{ac^{-1}} = \emptyset$. Thus, for any $\sigma \in P\left(\overline{ab^{-1}}\right)$ and $\tau \in P\left(\overline{ac^{-1}}\right)$, there exists a $k$ for which $\sigma(k) \neq \tau(k)$. Then $x_{\sigma} x_{\tau} = 0$, which implies that $\varphi_H(x_{\bar{a}\bar{b}})\varphi_H(x_{\bar{a}\bar{c}}) = 0$. Similarly $\varphi_H(x_{\bar{b}\bar{a}})\varphi_H(x_{\bar{c}\bar{a}}) = 0$, and so $\varphi_H$ extends uniquely to an $R$-algebra map $R[\MM_{(G^n\rtimes S_n)/(G^n\rtimes H)}] \rightarrow R[\MM_n]$. To prove it is a bialgebra map, it suffices to show that $\varphi_H$ preserves comultiplication on the generators. On the one hand, 
\begin{align*} 
\displaystyle \Delta (\varphi_H(x_{\bar{a}\bar{b}})) & = \Delta\left( \sum_{\sigma \in P\left(\overline{ab^{-1}}\right)}{x_{\sigma}} \right) \\ 
& =  \sum_{\sigma \in P\left(\overline{ab^{-1}}\right)}{\Delta(x_{\sigma})}\\ 
& = \sum_{\sigma \in P\left(\overline{ab^{-1}}\right)}{\prod_{k=1}^n{\Delta(x_{\sigma(k)k})}} \\ 
& = \sum_{\sigma \in P\left(\overline{ab^{-1}}\right)}{\prod_{k=1}^n{\left( \sum_{i_k}{x_{\sigma(k)i_k}\otimes x_{i_k k}} \right)}} \\ 
& = \sum_{\sigma \in P\left(\overline{ab^{-1}}\right)}{\sum_{i_1,\ldots i_n}{x_{\sigma(1)i_1}\cdots x_{\sigma (n)i_n} \otimes x_{i_1 1}\cdots x_{i_n n}}}
\end{align*} 
where the $i_1,\ldots , i_n$ are a permutation of $1,\ldots , n$. In other words, this can be written as  
\begin{align*} 
\displaystyle \sum_{\sigma \in P\left(\overline{ab^{-1}}\right)}{\sum_{i_1,\ldots i_n}{x_{\sigma(1)i_1}\cdots x_{\sigma (n)i_n} \otimes x_{i_1 1}\cdots x_{i_n n}}} & = \sum_{\sigma \in P\left(\overline{ab^{-1}}\right)}{\sum_{\tau \in S_n}{x_{\sigma \tau^{-1}}\otimes x_{\tau}}}\\ 
& = \sum_{\alpha \beta \in P\left(\overline{ab^{-1}}\right)}{x_{\alpha}\otimes x_{\beta}},
\end{align*} 
where the last sum is taken over all $\alpha, \beta \in S_n$ with $\alpha\beta \in P\left(\overline{ab^{-1}}\right)$. On the other hand, 
\begin{align*} 
\displaystyle [\varphi_H\otimes \varphi_H]\left( \Delta (x_{\bar{a}\bar{b}}) \right) & = [\varphi_H\otimes \varphi_H]\left(\sum_{\bar{c}}{x_{\bar{a}\bar{c}}\otimes x_{\bar{c}\bar{b}}} \right) \\ 
& = \sum_{\bar{c}}{\varphi_H(x_{\bar{a}\bar{c}})\otimes \varphi_H(x_{\bar{c}\bar{b}})} \\  
& = \sum_{\bar{c}}{\left( \sum_{\sigma \in P\left(\overline{ac^{-1}}\right)}{x_{\sigma}} \right) \otimes \left( \sum_{\tau \in P\left(\overline{cb^{-1}}\right)}{x_{\tau}} \right)} \\ 
& = \sum_{\bar{c}, \sigma, \tau}{x_{\sigma}\otimes x_{\tau}} \\ 
& = \sum_{\alpha \beta \in P\left(\overline{ab^{-1}}\right)}{x_{\alpha}\otimes x_{\beta}} \\ 
& = \Delta\left( \varphi_H(x_{\bar{a}\bar{b}}) \right).
\end{align*} 
To show that $\varphi_H$ is a $\hat{G}$-linear representation, take any $R$-algebra map $m : R [\MM_n]\rightarrow R$ satisfying $m(x_{ij}) \in \hat{G}$ for all $1 \le i,j \le n$. Then  
\begin{align*} 
m\left( \varphi_H ( x_{\bar{a}\bar{b}}) \right) & = m\left( \sum_{\sigma \in P\left(\overline{ab^{-1}}\right) }{x_{\sigma}} \right) \\ 
& = \sum_{\sigma \in P\left(\overline{ab^{-1}}\right) }{m(x_{\sigma})} \\ 
\end{align*} 
Now, $\displaystyle m(x_{\sigma}) = \prod_{k=1}^n{m(x_{\sigma(k)k})} \in \hat{G}$ for all $\sigma \in S_n$, and $m(x_{\sigma}) \neq 0$ if and only if $m(x_{\sigma(k)k}) \neq 0$ for all $1 \le k \le n$. Suppose such a $\sigma$ exists, and pick any $\tau \in S_n$ with $\sigma \neq \tau$. Then there exists a $k$ satisfying $\sigma (k) \neq \tau(k)$, and so $0 = m(0) = m(x_{\sigma(k)k}x_{\tau(k)k}) = m(x_{\sigma(k)k})m(x_{\tau(k)k})$, which implies $m(x_{\tau}) = 0$ since $m(x_{\sigma(k)k}) \in G$ is a unit. In other words, there is at most one $\sigma \in S_n$ for which $m(x_{\sigma}) \neq 0$. It thus follows that $[m\circ \varphi_H](x_{\bar{a}\bar{b}}) \in \hat{G}$ for all $\bar{a}, \bar{b} \in S_n/H$.
\end{proof}  

For any finite sets $S$ and $T$, there is a monic natural transformation $\MM_S\times \MM_T \rightarrow \MM_{S\sqcup T}$ corresponding to the block-diagonal embedding $(A,B) \mapsto \left(\begin{array}{cc} A & 0 \\ 0 & B \end{array} \right)$. Thus, if $\rho_i : R[\MM_{d_i}] \rightarrow R[\MM_n]$ are $\hat{G}$-representations of $\MM_n$ for $i=1,2$ then we have a natural transformation 
\[ 
\MM_n \rightarrow \MM_{d_1}\times \MM_{d_2} \hookrightarrow \MM_{d_1+d_2}
\]
whose dual map 
\[ 
R[\MM_{d_1+d_2}]\rightarrow R[\MM_{d_1}]\otimes_R R[\MM_{d_2}] \rightarrow R[\MM_n]
\] 
is easily checked to be a $\hat{G}$-linear representation of $\MM_n$. This corresponds to the \emph{direct sum} $\rho_1\oplus \rho_2$. In particular, $\Rep(\MM_n,\hat{G})$ has infinitely-many non-isomorphic objects.

%%%%%%%%%( THE FUNCTORIAL PERSPECTIVE )%%%%%%%%%%%%%%%%%%%%%%
\subsection{Functors on Abelian Groups}  

Recall that $\operatorname{Mon}_0$ denotes the category of monoids with absorbing elements, whose morphisms are monoid homomorphisms preserving absorbing elements. We let $\mon_0$ denote the full subcategory of finite monoids with absorbing elements, and $\hat{G}$-$\mon_0$ the subcategory of $\mon_0$ whose objects are $\hat{G}$-linear finite monoids, with $f : M \rightarrow N$ a morphism in $\hat{G}$-$\mon_0$ if and only if $f$ is a morphism in $\mon_0$ and $f(gx) = gf(x)$ for all $g \in G$ and $x \in M$. Then $\hat{G}$-$\mon_0$ contains the full subcategory $\hat{G}$-$\ab_0$ whose objects are finite pointed abelian groups containing $\hat{G}$. % We consider also the following full subcategories of $\Mon_0$:  
 %\begin{enumerate} 
% \item $\mon_0$, the full subcategory whose objects are finitely-generated. 
% \item $\cmon_0$, the full subcategory of $\mon_0$ whose objects are commutative. 
% \item $\ab_0$, the full subcategory of $\cmon_0$ consisting of the objects $M$ with the property that $M\setminus\{0_M\}$ is a finitely-generated abelian group. 
% \end{enumerate}

Assume $\hat{G}$ is fixed. Then the assignment $\hat{H} \mapsto I_n(\hat{H})$ defines a functor $\mfI_n(\hat{G}): \hat{G}$-$\ab_0 \rightarrow \hat{G}$-$\mon_0$. A natural transformation $\rho: \mfI_n(\hat{G}) \rightarrow \mfI_d(\hat{G})$ is equivalent to specifying, for each $\hat{H} \in \hat{G}$-$\ab_0$, a morphism
\[ 
\rho_H: I_n(\hat{H}) \rightarrow I_d(\hat{H})
\]  
in $\hat{G}$-$\mon_0$ such that for any morphism $\hat{H} \xrightarrow[]{\phi} \hat{K}$ in $\hat{G}$-$\ab_0$, we have a commutative diagram   

\[\begin{tikzcd}
I_n(\hat{H}) \arrow{r}{\rho_H} \arrow[swap]{d}{I_n(\phi)} & I_d(\hat{H}) \arrow{d}{I_d(\phi)} \\ 
I_n(\hat{K}) \arrow[swap]{r}{\rho_K} & I_d(\hat{K}). \\
\end{tikzcd}
\] 

\begin{mydef} 
A \emph{$\hat{G}$-linear representation} of $\mfI_n(\hat{G})$ of dimension $d$ is a natural transformation  
\[
\rho: \mfI_n(\hat{G}) \rightarrow \mfI_d(\hat{G})
\]
\end{mydef}   

\begin{mythm}\label{t.FunctorialRep}
Set $R = \ZZ[G]$ and $\MM_n = \MM_n(\hat{G})$. Suppose that $\rho : R[\MM_d] \rightarrow R[\MM_n]$ is a $\hat{G}$-linear representation of $\MM_n$ and that for all $1 \le k,l \le d$, $\rho(x_{kl})$ is of the following form: there exists a collection of injective functions $F(k,l)$ whose domains and ranges are subsets of $[n]$, satisfying\footnote{ Of course, the $\hat{G}$-representations described in Theorem \ref{t.newexamples} are of this form.}
\begin{enumerate}
\item Either $|F(k,l)|=1$ or $F(k,l) \subseteq S_n$  
\item $\displaystyle \rho\left(x_{kl} \right) = \sum_{f \in F(k,l)}{\lambda_fx_f}$, where if $f : S \rightarrow [n]$ then $\displaystyle x_f:=\prod_{i \in S}{x_{f(i)i}}$ and $\lambda_f \in G$. 
\end{enumerate}  
Then $\rho$ induces a $\hat{G}$-linear representation of $\mfI_n(\hat{G})$ of degree $d$.
\end{mythm}

\begin{proof}
 For all $\hat{H} \in \hat{G}$-$\ab_0$, set $R[\MM_n]_H = \ZZ[H]\otimes_{R}R[\MM_n]$. Note that $R[\MM_n]_H$ is the bialgebra representing $\MM(\hat{H})$. For each such $H$, $\rho$ induces a $\ZZ[H]$-algebra homomorphism 
\[ 
\rho_H : R[\MM_d]_H \rightarrow R[\MM_n]_H
\] 

via extension of scalars. Note that for any two $\hat{H},\hat{J} \in \hat{G}$-$\ab_0$, condition (2) implies that $\rho_H(x_{ij})$ and $\rho_J(x_{ij})$ may naturally be identified as elements of the common subring $R[\MM_n]$ of $R[\MM_n]_H$ and $R[\MM_n]_J$. In other words, $\rho_H(x_{ij}) = \rho_J(x_{ij})$ for all $i$ and $j$ as elements of $R[\MM_n]$. Suppose that $m : R[\MM_n]_H \rightarrow \ZZ[H]$ is a $R$-algebra homomorphism which satisfies $m(x_{kl}) \in \hat{H}$ for all $k$ and $l$. Then $\displaystyle  m(x_f) = \prod_{i \in S}{m(x_{f(i)i})} \in \hat{H}$ for any $f$, and condition (1) above ensures that $m(x_f) \neq 0$ for at most one $f \in F(k,l)$. Thus, condition (2) ensures that $[m\circ \rho_H](x_{kl}) \in \hat{H}$ for all $1\le k,l \le d$.  In other words, $\rho_H$ induces a monoid homomorphism $\tilde{\rho}_H: I_n(\hat{H}) \rightarrow I_d(\hat{H})$. We claim that $\tilde{\rho}_H$ is a morphism in $\hat{G}$-$\mon_0$. Indeed, suppose that $m : R[\MM_n]_H \rightarrow \ZZ[H]$ is a $\ZZ[H]$-algebra map corresponding to the matrix $A \in I_n(\hat{H})$ and that $g \in G$. Let $gm$ denote the $\ZZ[H]$-algebra map defined by $x_{ij} \mapsto gm(x_{ij})$ for all $1 \le i,j \le n$, which corresponds to $gA \in I_n(\hat{H})$.  Then for each $x_{kl}$, $\tilde{\rho}_H(gm)(x_{kl}) = [(gm)\circ \rho_H](x_{kl}) = g[m\circ \rho_H](x_{kl}) = g\tilde{\rho}_H(m)(x_{kl})$. In other words, $\tilde{\rho}_H(gA) = g\tilde{\rho}_H(A)$. Suppose that $f : \hat{H} \rightarrow \hat{J}$ is a morphism in $\hat{G}$-$\ab_0$. Let $m : R[\MM_n]_H \rightarrow \ZZ[H]$ be an element of $I_n(\hat{H})$. Then $\tilde{\rho}_H(m) = m\circ \rho_H \in I_d(\hat{H})$. In fact, the $(i,j)$-entry of $m\circ \rho_H$ is precisely $m(\rho_H(x_{ij}))$. Thus, the $(i,j)$-entry of $I_d(f)(\tilde{\rho}_H(m))$ is precisely $f(m(\rho_H(x_{ij})))$. On the other hand, $I_n(f)(m) \in I_n(\hat{J})$ and thus $I_n(f)(m)\circ \rho_J = \tilde{\rho}_J(I_n(f)(m)) \in I_d(\hat{J})$. The $(i,j)$-entry of this matrix is $[I_n(f)(m)]\rho_J(x_{ij}) = f(m(\rho_J(x_{ij}))) = f(m(\rho_H(x_{ij})))$, where the last equality follows from the identification described above. In other words, we have a commutative square 

\[ 
\begin{tikzcd}
 I_n(\hat{H}) \arrow{r}{\tilde{\rho}_H}\arrow{d}[swap]{I_n(f)} & I_d(\hat{H})\arrow{d}{I_d(f)}\\
 I_n(\hat{J}) \arrow{r}[swap]{\tilde{\rho}_J}& I_d(\hat{J})  \\
\end{tikzcd} 
\]
and thus, a $\hat{G}$-linear representation $\tilde{\rho} : \mfI_n(\hat{G}) \rightarrow \mfI_d(\hat{G})$.  
\end{proof}

\begin{mydef}
In analogy to our bialgebra framework, a {\emph{morphism}} from $\rho^{(1)} : \mfI_n(\hat{G}) \rightarrow \mfI_{d_1}(\hat{G})$ to $\rho^{(2)} : \mfI_n(\hat{G})\rightarrow \mfI_{d_2}(\hat{G})$ is a collection of morphisms in $\hat{G}$-$\ab_0$  
\[
f = \{f_H : \hat{H}^{\oplus d_1} \rightarrow \hat{H}^{\oplus d_2} \mid  \hat{H} \in \hat{G}-\ab_0 \} 
\] 
satisfying the following: 
\begin{enumerate}
 \item All of the corresponding diagrams 
\[
\begin{tikzcd}  
I_n(\hat{H})\times \hat{H}^{\oplus d_1} \arrow{r}{\rho^{(1)}_{\hat{H}}\times \id} \arrow[swap]{d}{\id\times f_H} & I_{d_1}(\hat{H}) \times \hat{H}^{\oplus d_1} \arrow{r}{\mu_1} & \hat{H}^{\oplus d_1} \arrow{d}{f_H} \\ 
I_n(\hat{H})\times \hat{H}^{\oplus d_2} \arrow[swap]{r}{\rho^{(2)}_H \times \id}  & I_{d_2}(\hat{H})\times \hat{H}^{\oplus d_2} \arrow[swap]{r}{\mu_2} & \hat{H}^{\oplus d_2}\\
\end{tikzcd} 
\]   
commute, where $\mu_1$ and $\mu_2$ are the maps affording the usual $\hat{H}$-linear actions.  
\item The maps $f_H : \hat{H}^{\oplus d_1} \rightarrow \hat{H}^{\oplus d_2}$ restrict to $\hat{G}$-linear maps $\hat{G}^{\oplus d_1} \rightarrow \hat{G}^{\oplus d_2}$ for each $\hat{H}$.
\end{enumerate} 
\end{mydef}  

With this definition, the collection of all $\hat{G}$-linear representations of $\mfI_n(\hat{G})$ becomes a category, denoted $\Rep(\mfI_n(\hat{G}),\hat{G})$. If $f : \rho^{(1)} \rightarrow \rho^{(2)}$ is an isomorphism, then setting $\hat{H} = \hat{G}$ implies that $\rho^{(1)}_G \cong \rho^{(2)}_G$ as objects in $\Rep(I_n(\hat{G}),\hat{G})$. In particular, Theorem \ref{t.FunctorialRep} and the results of the preceding section imply that $\Rep(\mfI_n(\hat{G}),\hat{G})$ has infinitely-many non-isomorphic objects.

\bibliography{quiver}\bibliographystyle{alpha} 
\end{document}